\documentclass[twoside,11pt]{book}
\usepackage[latin1]{inputenc}
\usepackage{amsthm}
\usepackage{amssymb}
\usepackage{amsmath}
\numberwithin{section}{chapter}
\theoremstyle{plain}
\theoremstyle{plain}
\newtheorem{thm}{Theorem}
\newtheorem*{thmB}{Theorem B}
\newtheorem*{thmC}{Theorem C}
\newtheorem*{conjD}{Conjecture D}
\theoremstyle{definition}
\newtheorem{defn}[thm]{Definition}
\theoremstyle{plain}
\newtheorem{prop}[thm]{Proposition}

\newtheorem*{propA}{Proposition A}
\newtheorem{lem}[thm]{Lemma}

\theoremstyle{remark}
\newtheorem{rem}[thm]{Remark}
\newtheorem{ex}[thm]{Example}

\newtheorem{question}[thm]{Question}
\usepackage{graphicx}
\usepackage{amscd}
\usepackage[all,knot,poly]{xy}
\usepackage[english]{babel}
\usepackage{mathrsfs}
\usepackage{stmaryrd}
\usepackage{hyperref}
\usepackage{fancyhdr}
\usepackage{titlesec}
\titleformat
{\section} % command
[display] % shape
{\mdseries\LARGE\scshape} % format
{} % label
{0ex} % sep
{\centering\thesection\hspace*{1em}}
\titlespacing{\section}{0pt}{4ex}{2ex}
\newcommand{\C}{\mathbb{C}}
\newcommand{\R}{\mathbb{R}}
\newcommand{\Q}{\mathbb{Q}}
\newcommand{\Z}{\mathbb{Z}}
\newcommand{\N}{\mathbb{N}}

\newcommand{\OO}{\mathcal{O}}
\newcommand{\F}{\mathcal{F}}
\newcommand{\G}{\mathcal{G}}
\newcommand{\all}{\mathcal{A}ll}
\newcommand{\s}{{Section}}
\newcommand{\stab}{\mathrm{stab}}
\newcommand{\id}{\trianglelefteq}
\newcommand{\ox}{^\circledast}
\newcommand{\dx}{^\times}
\newcommand{\sx}{^\times\!}
\newcommand{\Ext}{\mathrm{Ext}}
\newcommand{\ind}{\mathrm{ind}}
\newcommand{\res}{\mathrm{res}}
\newcommand{\spn}{\mathrm{span}}
\newcommand{\kom}{\mathfrak{K}om}
\newcommand{\gen}[1]{\langle #1\rangle}
\newcommand{\fr}[1]{\mathfrak{#1}}
\newcommand{\ul}[1]{\underline{#1}}
\newcommand{\slot}{\,\underbar{\;\;}\,}
\DeclareFontFamily{U}{wncy}{}
    \DeclareFontShape{U}{wncy}{m}{n}{<->wncyr10}{}
    \DeclareSymbolFont{mcy}{U}{wncy}{m}{n}
    \DeclareMathSymbol{\Sh}{\mathord}{mcy}{"58}
    
\newenvironment{changemargin}[2]{%
 \begin{list}{}{%
  \setlength{\topsep}{0pt}%
  \setlength{\leftmargin}{#1}%
  \setlength{\rightmargin}{#2}%
  \setlength{\listparindent}{\parindent}%
  \setlength{\itemindent}{\parindent}%
  \setlength{\parsep}{\parskip}%
 }%
\item[]}{\end{list}}
\author{Federico William Pasini}
\title{Knots and primes}
\makeindex
\begin{document}
\frontmatter
%----------------------------------------------------------------------------------------
%	TITLE PAGE
%----------------------------------------------------------------------------------------

\begin{titlepage}
\begin{changemargin}{3.6em}{-3.6em}
\begin{center}

{\scshape\LARGE \par}\vspace{2.5cm}
\textsc{\Large Doctoral Thesis}\\[0.5cm]

{\huge \bfseries CLASSIFYING SPACES FOR KNOTS\\ New bridges between knot theory and algebraic number theory\par}\vspace{1cm}
\LARGE
\textsc{Federico William Pasini}\\[1cm]
\textsc{\emph}{Supervisor:} \\
\textsc{Professor Thomas Stefan Weigel}
\\[3cm]
 
\Large \textit{A thesis submitted in partial fulfillment of the requirements\\ for the degree of Philosophiae Doctor in Mathematics}\\[4cm]
\large
Dipartimento di Matematica e Applicazioni\\Universit\`a degli Studi di Milano-Bicocca\\[0.5cm]
 
{\large 13 September 2016}
 
\vfill
\end{center}
\end{changemargin}
\end{titlepage}
\vspace*{7cm}
\noindent{\LARGE\textit{To my brother-in-soul\\Dario Merlin}}
\vfill

\titleformat
{\chapter} % command
[display] % shape
{\mdseries\huge\scshape} % format
{} % label
{0ex} % sep
{
    \centering
}
[
\rule{\textwidth}{1pt}
\vspace{1ex}
]
\cleardoublepage
\addcontentsline{toc}{chapter}{Preface}
\chapter*{Preface}
This thesis is aimed at earning the degree of \emph{Philosophiae Doctor}, so I find it good to start it with some philosophy.

Regarding epistemology, I am strictly antimetaphysical.
I believe the epistemological essence of a human being is the datum of all his or her relationships with the others.
In this respect, I owe not only my gratitude, but also a fragment of my essence (which I am quite happy with), to all the people that shared a bit of their road with me.

But I am also an environmentalist, and thanking explicitly all the people I should would cause a too grievous sapshed among the poplar community.

Then, following an idea of Dr. Dario Celotto, I will cite here only those who had a direct influence on my PhD studies.

My gratitude goes to Thomas Weigel, who has always been more than just a mathematical advisor. He instilled the difference between good mathematics and formal manipulation of symbols in me, and constantly exhorted me to bear the pains of the path of good mathematics. He urged me to work on my weaknesses. He offered all his students periodic psychological sessions. He taught me the basics of tango. Above all, he helped me clarify my thoughts in a period of personal dire straits and guided me towards a better approach to my life as a whole. I do not know if this behaviour is general among PhD supervisors, but his case is sufficient to justify the meaningfulness of the Mathematics Genealogy Project\linebreak {\verb|https://genealogy.math.ndsu.nodak.edu/|}.

I also had the chance of spending some months as a visitor at Royal Holloway University of London. There I  was adopted by my foster advisor Brita Nucinkis, and the choice of words is not exaggerated. For her quietness, her patience and her maternal instinct, as I should say, made me feel right at home. As a consequence, I had a great time, both personally and professionally, with her and her PhD students, Ged Corob Cook and Victor Moreno. I strongly hope that period was just the beginning of a lifelong cooperation and friendship. I also benefited from some kind advice and feedbacks from Iain Moffatt. And how could I not mention the tuscan delegation at Royal Holloway? I will always bring a little of Eugenio Giannelli and Matteo Vannacci's contagious enthusiasm and vivacity with me.

Being a researcher cannot be only a job. It is a kind of approach to the whole reality. And it is a fantastic chance of growth when a wise senior researcher shares his vision with you. That is the reason why I am grateful to Roberto La Scala from the University of Bari and Ugo Moscato and Renzo Ricca from the University of Milano-Bicocca.

I left to the end the people my story braided the most with, my nest fellows at the university of Milano-Bicocca.
My elder brothers Tommaso Terragni, Claudio Quadrelli and Gianluca Ponzoni have always been ready to provide me with encouragement in hard times. Above all, they built a sense of team group among PhD students in algebra that survived their departure. I hope I managed to even strengthen this spirit and to extend it, beyond the boundaries of algebra, to all the younger PhD students. Whether I succeeded might be checked in the acknowledgements of their theses.

I shared all the (many) joys and (many) sorrows of the doctoral studies with Dario Celotto, Sara Mastaglio, Elena Rossi, Marina Tenconi, Elisa Baccanelli, Valentina Folloni, Bianca Gariboldi, Iman Mehrabi Nezhad, Paola Novara, Elena Volont\'e, Alessandro Calvia, Elia Manara, Arianna Olivieri, Mariateresa Prandelli, Martino Candon, Jessica Cologna, Andrea Galasso, Jinan Loubani, Daniela Schenone and Alberto Viscardi. The Risiko nights, the hikes in the mountains, the trips and the dinners we enjoyed together will be forever on my mind.
Finally, Ilaria Castellano and Benedetta Lancellotti have become my best friends by virtue of all the adventures we have gone through together (and in which citrus fruits usually play a big role). They would deserve a praise that cannot be expressed in words, but only with a big hug.

Without them all, I would have probably produced far more mathematics, and had far less fun.
\tableofcontents
\mainmatter
\titleformat
{\chapter} % command
[display] % shape
{\mdseries\huge\scshape} % format
{\centering Chapter \thechapter} % label
{0ex} % sep
{
    \rule{\textwidth}{1pt}
    \vskip1.2ex
    \centering
} % before-code
\chapter{Introduction}
It may very well be claimed that the birth certificate of topology is Leonhard Euler's solution to the \emph{bridges of K\"onigsberg} problem, in 1735. Curiously, that same year saw the birth of Alexandre-Teophile Vandermonde, who was the first to explicitly conceive the possibility of a mathematical study of knots and links as a part of the new discipline (cf. \cite{vandermonde}). This idea spreaded silently in the mathematical community of the 18th century, but it was not until the 19th century that Carl Friedrich Gauss laid the first two foundational stones of knot theory, which would become a source of inspiration for plenty of later developments. On one side, he developed a first method to tabulate the knots, by writing the sequence of crossings met by a point running on a knot. On the other, after becoming involved in electrodynamics, he introduced the notion of linking number of two disjoint knots, i.e., the number of times the first knot winds around the other (probably in connection with Biot-Savart law, cf. \cite{ricca}):
\begin{equation}\label{eq:lk}
 \mathrm{Lk}(K_1,K_2)=\frac{1}{4\pi}\oint_{K_1}\oint_{K_2}\frac{r_1-r_2}{|r_1-r_2|^3}(dr_1\times dr_2).
\end{equation}
He also proved that this linking number remains the same if the knots are interchanged and that it is invariant under ambient isotopy: it is the first example of a link type invariant. 

Classical knot theory deals with the embeddings $\sqcup S^1\hookrightarrow S^3$ of (collections of) circles into the $3$-sphere (classical links). The significance of the theory is easily explained. On one side, links are sufficiently tangible objects (as far as a mathematical object can be) to provide a comfortable environment for testing new techniques and developing new ideas. On the other side, $3$-dimensional manifolds enjoy a good deal of features that do not appear in any other dimension and make dimension $3$ particularly interesting (and challenging, as the history of Poincar\'e conjecture showed) to topologists (see e.g. the survey \cite{topo3mani}), and links play a prominent role in the topology of $3$-manifolds: for example, J. W. Alexander showed (cf. \cite{alexander}) that every closed orientable $3$-manifold is a branched covering of $S^3$ branched over a link.
\vspace{\baselineskip}

Algebraic number theory is a discipline that develops from the very roots of both ancient algebra (the theory of equations) and modern algebra (the theory of structures, after \'Evariste Galois). It ideally started with diophantine equations (cf. Diophantus's \textit{Arithmetica}). The search for solutions to such equations led to the development of the concepts of divisibility and primality, which in turn attracted the attention of many great mathematicians of 17th and 18th century, like Pierre de Fermat, Euler, Joseph-Louis Lagrange and Adrien-Marie Legendre. But it was again Gauss that gathered all the previous isolated results into a stable and systematic theory. And, as he was Gauss for a good reason, he also added some of his personally crafted jewels into his foundational book \textit{Disquisitiones arithmeticae} (see \cite{gauss} for an english translation).
In particular, he proved the \emph{quadratic reciprocity law}: for distinct odd prime numbers $p,q\in\N$, 
\[
 \left(\frac pq\right)\left(\frac qp\right)=(-1)^{(p-1)(q-1)/4},
\]
\[
\mbox{where }\:\left(\frac pq\right)=\begin{cases}
                                           1 & p\mbox{ is a square mod }q\\
                                           -1 & \mbox{otherwise}
                                          \end{cases}\quad\mbox{ is called the \emph{Legendre symbol}.}
\]
In trying to generalise this theorem to higher powers than squares, Gauss proved that the elements of $\Z[i]$ enjoy an essentially unique factorisation into primes, like the ``classical'' integers do, and the failure of a general $\Z[\sqrt{-n}]$ to be a unique factorisation domain eventually led Ernst Kummer and Richard Dedekind to formulate the concept of ideal of a ring. This in turn paved the way to the realisation that number theory could be conveniently reformulated in terms of ``integer numbers'' in finite extensions of $\Q$. And that is the story of how algebraic number theory intermarried with Galois theory. In modern terms, algebraic number theory generalise the notion of \emph{integer number} in the field of rationals to that of \emph{algebraic integer} in a finite field extension of $\Q$, and deals with the relationship between prime numbers and prime ideals in the rings of algebraic integers of such extensions.
\vspace{\baselineskip}

Along the years, knot theory and algebraic number theory have grown up independently, but apparently they have kept track of a common germ. M. Morishita collected, in his pleasant book \cite{morishita}, lots of interesting analogies between knot theory and algebraic number theory, based on a sort of ``geometrization of numbers'' that highlights the similar behaviour of knots on one side and primes on the other. Remarkably, one of these analogies matches the linking number with the Legendre symbol, and the integral \ref{eq:lk} expressing the former with a Gauss sum yielding the latter, also proved by Gauss
\[
 \left(\sum_{x\in F_q}\zeta_q^{x^2}\right)^{p-1}=\left(\frac{(-1)^{\frac{q-1}{2}}q}{p}\right).
\]

Might it be that the prince of mathematicians already had this vision of magic bridges that guided him during his explorations in the realms of topology and number theory?
\vspace{\baselineskip}

It is my intent to carry over this line of research by exploiting a new topological ingredient that has recently entered the scene of mathematics, but, up to my knowledge, has never been conceived in this framework, i.e. classifying spaces for families of subgroups. These spaces generalise the concept of total space $EG$ of the universal principal bundle $EG\to BG$ for the group $G$. More precisely, they relax the absence of $G$-fixed points to the requirement that only a prescribed family of subgroups of $G$ fix points and the fixed-point space of each such subgroup be contractible. Classifying spaces gained the attention of mathematicians in connection with the Baum-Connes Conjecture on the topological $K$-theory of the reduced group $C^*$-algebra (cf. \cite{misval}) and the Farrell-Jones Conjecture on the algebraic $K$- and $L$-theories of group rings (cf. \cite{farjon}). 

Classifying spaces are known to exist and to be unique up to $G$-homotopy for all families of all groups $G$ (cf. \cite[Subsection 1.2]{luecksurvey}). Yet, a major problem is to find nice concrete models for such spaces, on which to carry effective computations. 

In the context of knot theory, a classifying space of a knot group $G$ for the family of subgroups generated by meridians can be defined. It amounts to ``completing'' the universal cover of the knot complement (which carries a free $G$-action) to a $G$-space such that the quotient over $G$ is the whole $S^3$. As a first result, I present a construction that provides explicit models of these classifying spaces in the case of prime knots. In detail, if $G$ is a prime knot group and $H\cong\Z^2$ is a peripheral subgroup generated by a meridian $a$ and a longitude $l$ of $G$ (see Chapter \ref{chp:knotTheory} for the definitions), then the classifying space $E_\G G$ of $G$ for the family of subgroups generated by meridians is given by a pushout of $G$-CW-complexes
\begin{equation*}
\xymatrix{
\bigsqcup_{G/H}\R^2\ar^\phi[rrr]\ar^{id_G\times\psi}[d] &&& EG\ar[d]\\
\bigsqcup_{G/H}Cyl(\pi\colon\R^2\to\R)\ar[rrr]&&&E_\G G.
}
\end{equation*}
Here $Cyl(\pi\colon\R^2\to\R)$ denotes the mapping cylinder of the canonical projection from $EH\approx\R^2$ to the subcomplex $EH^{\gen a}\approx\R$ of points fixed by the subgroup $\gen a$ and the maps starting from the upper-left corner glue the term $EH$ in each copy of that mapping cylinder to a boundary component of $EG$.

Actually, an analogous pushout does make sense for non-prime knots as well, as it is explained in \S \ref{ssec:non-prime knots}, but the attaching maps, and thus the model $E_\G G$ obtained, are fairly more complicated. In other words, in the case of non-prime knots the models lose those features of neatness that make their prime knot analogues attractive.

The classifying spaces just constructed can serve to extend and to shed a new light on the topic of branched coverings of knots. In this respect, having an explicit model for the classifying space enables to prove the following.
\begin{propA}
Let $G$ be a prime knot group and $U\id G$ a normal subgroup of finite index. There is a short exact sequence of groups
\begin{equation*}
\xymatrix{
1\ar[r]&M_U\ar[r]&U\ar[r]&\pi_1(E_\G G/U)\ar[r]&1,
}
\end{equation*}
where $M_U$ is the normal subgroup of $U$ generated by those powers of the meridians of the knot that stay in $U$.
\end{propA}
This sequence is interesting on its own. Indeed, applying the abelianisation functor we obtain the exact sequence
\[
\xymatrix{
 M_U^{ab}\ar[r]&U^{ab}\ar^-{\pi}[r]&H_1(E_\G G/U)\ar[r]&0.
 }
\]
And now let us put on the number theory glasses. Take an algebraic number field $K$, call $\mathcal O_K$ its ring of algebraic integers and $G_K$ its absolute Galois group, i.e. the Galois group of its separable closure (cf. \cite[Section IV.1]{neukirch}).
A direct consequence of Hilbert class field theory (cf. \cite{neukirch}) is that the ideal class group $Cl_K$ (cf. Chapter \ref{chp:ant}) fits into an exact sequence
\[
 \xymatrix{
 \coprod_{p\in \mathrm{Spec(\OO_K)}}\OO_p^*\ar[r]&G_K^{ab}\ar^-{\varpi}[r]&Cl_K\ar[r]&0,
 }
\]
where $\OO_p$ is the complete discrete valuation ring associated with the prime $p\id\OO_K$.
Even if an explicit description has not been obtained yet, it is known that $\ker(\varpi)$ is generated by the ramification groups (cf. \cite[Section II.9]{neukirch}) of the primes of $K$. On the other hand, $\ker(\pi)$ is generated by the ``meridians'' of $BU$, which, as we will see in Chapter \ref{chp:classSpaces}, are responsible of the ramification in branched coverings of knots. In other words, $\varpi$ kills the ramification of prime ideals in algebraic number fields, exactly like $\pi$ ``kills'' the ramification of branched coverings of $S^3$.

But the sequence is also a key tool in proving the following 
\begin{thmB}
Let $G$ be a prime knot group and let $U\id G$ be a normal subgroup of finite index. Then the following are equivalent.
\begin{enumerate}
\item $U=M_U$;
\item The canonical projection $E_\G G/M_U\to E_\G G/U$ is a trivial covering;
\item The canonical projection $E G/M_U\to E G/U$ is a trivial covering;
\item $\pi_1(E_\G G/U)=1$;
\item $E_\G G/U\cong S^3$;
\end{enumerate}
\end{thmB}

A conjectural relationship between the (co)homology of the classifying spaces and the Shafarevi\v{c} groups of algebraic number fields inspired the discovery of a Poitou-Tate-like $9$-term exact sequence. In detail, if $S$ is a nonempty set of primes of an algebraic number field $K$, including the primes at $\infty$, and $A$ is a finite $G_S$-module s.t. $\nu(|A|)=0$ for all $\nu\notin S$, Poitou-Tate sequence is the $9$-term exact sequence
\begin{equation*}
\xymatrix{
0\ar[r] & H^0(K_S,A)\ar[r] & P^0(K_S,A)\ar[r] & H^2(K_S,A')^\vee\ar[dll]&\\
& H^1(K_S,A)\ar[r] & P^1(K_S,A)\ar[r] & H^1(K_S,A')^\vee\ar[dll]&\\
& H^2(K_S,A)\ar[r] & P^2(K_S,A)\ar[r] & H^0(K_S,A')^\vee\ar[r]&0\\
}
\end{equation*}
connecting the Galois cohomology of $K$ with the Galois cohomology of its completions with respect to a set of non-archimedean distances induced by the prime ideals of the ring of algebraic integers of $K$ (see Chapter \ref{chp:ant} for explanations).
The kernel of the map $\beta$ is the \emph{first Shafarevi\v{c} group} of $K$:
\[
\Sh^1(G_S,A)=\ker\Big(H^1(K_S,A)\to P^1(K_S,A)\Big).
\]

By using tools from the theory of derived categories with duality, a brief review of which can be found in the appendix, it has been possible to obtain an amazingly similar sequence for knot groups, taking into account that
\begin{itemize}
 \item this sequence is in homology, so all arrows go in the opposite direction with respect to those in Poitou-Tate sequence;
 \item $H^i(K_S,A')^\vee$ in Poitou-Tate sequence stands for the Pontryagin dual of $H^i(K_S,A')$, and the cohomology groups appearing in our sequence are morally dual to their homology counterparts.
\end{itemize}
\begin{thmC}
Let $G$ be a knot group, $H\leq G$ a peripheral subgroup and $U\id G$ a normal subgroup of finite index. Then there is an exact sequence of groups
\begin{equation}\label{eq:poitou-tate U intro}
\xymatrix{
0\ar[r] & H^0(U,\Z)\ar[r] & \coprod_{G/U\! H} H_2(H\cap U,\Z)\ar[r] & H_2(U,\Z)\ar[lld] & \\
 & H^1(U,\Z)\ar[r] & \coprod_{G/U\! H} H_1(H\cap U,\Z)\ar[r] & H_1(U,\Z)\ar[lld] & \\
 & H^2(U,\Z)\ar[r] & \coprod_{G/U\! H} H_0(H\cap U,\Z)\ar[r] & H_0(U,\Z)\ar[r] & 0.
}
\end{equation}
\end{thmC}
And, if we cut out the subsequence in degree $1$,
\begin{conjD}
Let $G$ be a knot group, $H\leq G$ a peripheral subgroup and $U\id G$ a normal subgroup of finite index. Then there is an exact sequence of groups
\[
\xymatrix@C-6pt{
0\ar[r] & H^1(E_\G G/U,\Z)\ar[d]\\
 & H^1(U,\Z)\ar[r] & \coprod_{G/U\! H} H_1(H\cap U,\Z)\ar[r] & H_1(U,\Z)\ar[d] & \\
 &&&H_1(E_\G G/U,\Z)\ar[r] & 0.
}
\]
\end{conjD}
Note that $H\cap U$ is still isomorphic to $\Z^2$, as $U$ has finite index in $G$. Its generators are products of powers (in multiplicative notation) of the generators of $H$.

Conjecture D will be proved in some special cases, but remains open in full generality, so the hunt is not over.

Besides the charming analogies, there is also a remarkable difference between knots and numbers. That is, in the realm of knots the first homology and the first cohomology group of the same space $E_\G G/U$ appear as candidates to complete the degree-$1$ subsequence of \eqref{eq:poitou-tate U intro} on both ends. This has no parallel in algebraic number theory.
Knot theory has topological spaces naturally available to work with, while algebraic number theory has not. And dealing with groups that come as the (co)homology of such spaces might be far easier than dealing with generic groups, since homology and cohomology are bonded to each other by plenty of duality principles.

An explicative example is Leopoldt Conjecture in number theory. Let $K$ be an algebraic number field with $[K:\Q]=n$, and suppose for simplicity that $K\supseteq\Q[i]$ (so that $K$ has no real embeddings) and $K$ is ``large enough'' (e.g., it contains a primitive $p$th root of unity, for some prime $p\in\Z$). Set $S_p=\{Q\id\OO_K\mid Q\cap\Z= p\}$ (i.e., $Q$ lies above $p$, cf. Section \ref{sec:ramification of ideals}), $S_\infty$ the set of places of $K$ at infinity (cf. Section \ref{sec:valuations}) and $S=S_p\cup S_\infty$ (note that this is a finite set). Furthermore, let $D_Q$ be the decomposition group of $Q$ (cf. \cite[Section I.9]{neukirch}) and let $G_K^S$ be the Galois group of the maximal field extension of $K$ unramified outside $S$ (cf. \cite[Section III.2]{neukirch}).
Then there is an exact sequence
\[
\xymatrix{
 H^1\left(G_K^S,\Z_p(1)\right)\ar[r] & \coprod_{Q\in S}H_1\left(D_Q,\Z_p\right)\ar[r] & H_1\left(G_K^S,\Z_p\right)
 }
\]
(for the meaning of $\Z_p(1)$, cf. \cite{quadro}).
The Leopoldt Conjecture claims that if one completes this sequence to
\[
\xymatrix{
0\ar[r] & T_1\ar[d]&&&\\
& H^1\left(G_K^S,\Z_p(1)\right)\ar[r] & \coprod_{Q\in S}H_1\left(D_Q,\Z_p\right)\ar[r] & H_1\left(G_K^S,\Z_p\right)\ar[d]&\\
&&& T_2\ar[r] & 0
 }
\]
then the groups $T_1$ and $T_2$ are torsion groups. There are cases (e.g., when one passes to the maximal pro-$p$ quotient $G_K^S(p)$ of $G_K^S$) in which $T_2$ is known to be even finite, but this still does not give any hints on $T_1$. On the contrary, in the cases in which Conjecture D is true, the finiteness of $H_1(E_\G G/U,\Z)$ would immediately imply the triviality of $H^1(E_\G G/U,\Z)$, via the Universal coefficient Theorem for cohomology (cf. \cite[Section 3.1]{hatcher})!

But I am an optimistic guy. Maybe one day some talented mathematician will carry forward the ``geometrization of numbers'' started by Gauss and find a way to ``spacify'' the cohomology of number fields, thus restoring the lost symmetry. I hope I will be nearby. For, I am sure, what will be disclosed will be truly amazing.

\section{Notation}
$\N$ denotes the set of natural numbers \emph{including} $0$. $\N^*=\N\setminus\{0\}$.
An isomorphism between sets with algebraic structure will be denoted $\cong$.
In a group $G$, $\gen{g_1,\dots,g_n}$ is the subgroup generated by $g_1,\dots,g_n$, while $\gen{\gen{g_1,\dots,g_n}}$ is the \emph{normal} subgroup generated by $g_1,\dots,g_n$. In a ring $R$, $R^*$ is the set of invertible elements.

$S^n$ is the $n$-dimensional sphere, $B^n$ the (open) $n$-dimensional ball.
In the category of topological spaces, $\approx$ denotes a homeomorphism, while $\simeq$ denotes a homotopy equivalence. If $Y$ is a topological space, $\overline{Y}$, $\partial Y$, $Y^\circ$ stand for its closure, boundary and interior, respectively.

\chapter[Knot theory]{Knot theory in a nutshell}\label{chp:knotTheory}
General references for this chapter are \cite{burzie} and \cite{crofox}.

\section{Basic definitions}
A \emph{$c$-component link}, or simply \emph{$c$-link} ($c\in\N^*)$ is a topological embedding, i.e. a homeomorphism onto its image, $\lambda:S^1\times\{1,\dots,c\}\longrightarrow S^3$, from a disjoint union of circles into the $3$-sphere. We usually identify it with its image $L$. The \emph{components} of the link are $L^{(i)}:=\lambda(S^{1}\times \{i\})$ for $i=1,\dots,c$. A \emph{knot} is a $1$-link.
\begin{rem}
In the most naive setting of knot theory, the codomain of a link should be the ``tangible'' space $\R^3$; the choice of $S^3$ instead is only a matter of comfort. In fact, $S^3$ is just the one-point compactification of $\R^3$, and working in a compact $3$-manifold does bring some advantages. Nevertheless, we shall occasionally feel free to view a link as an embedding in a copy of $\R^3$ sitting inside $S^3$.
\end{rem}

Thinking of circles as parametrized curves $\gamma\colon[0,2\pi]\to\C, \gamma(t)=e^{it}$, we can choose an \emph{orientation} on each component of a link, that is, a direction of travel on that component.

There is a natural notion of equivalence among links, defined as follows.
An isotopic deformation of a topological space $X$ is a map $D=D(x,t):X\times [0,1]\longrightarrow X$, simultaneously continuous in both its arguments $x$ and $t$, such that $\forall t\in [0,1], d_t:=D(\cdot,t)$ is a homeomorphism and $d_0$ is the identity. Two links $L$ and $K$ are \emph{equivalent} if there is an isotopic deformation $D$ of $\R^3$ mapping the former onto the latter, namely $d_1(L)=K$. This yields a true equivalence relation whose classes are called \emph{link types}. In particular, a knot is \emph{trivial} if it is equivalent to the standard circle
\[
\{x^2+y^2=1, z=0\}\subset\R^3\hookrightarrow S^3,
\]
or informally, if it is ``unknotted''.

The general problem of knot theory is to understand whether two given knots (links) belong to the same type. In order to achieve this task, a number of invariants of the link type have been developed from algebraic topology. Among them, the link group holds a prominent role, but before giving its definition we shall circumscribe the class of links for which the algebraic invariants manage to express their full potential.

\emph{Polygonal} links are the ones whose components are the union of finitely many closed straight-line segments. A link is \emph{tame} if it has the same type of a polygonal one, \emph{wild} otherwise. The two terms also apply to link types at once.
Tame links can be satisfactorily handled with algebraic methods. Roughly speaking, this is so because  the behaviour of polygonal curves can be encoded with finitely many information and do not require, for example, any limit process (in the following, several precise instances of this sentence will appear clearly). That is why, from now on, \emph{all links are tacitly assumed to be tame}.

The \emph{link group} $G:=G_L$ of a link $L$ is the fundamental group of its complement in $S^3$, which is independent of the base point thanks to path-connectedness.
A presentation is obtained by the following procedure:
\begin{enumerate}
\item
a polygonal link (here viewed in $\R^3$) can be projected on a suitable plane in such a way that
\begin{enumerate}
\item[(a)] the projection is $1$ to $1$ at all points of the link, except possibly a finite number of \emph{crossing points}, where it is $2$ to $1$ (the images of such points are called \emph{crossings});
\item[(b)] no vertex of the polygonal curves is projected to a double point.
\end{enumerate}
Suppose without loss of generality that the aforementioned suitable plane is the $XY$ plane. Then at each crossing we distinguish an upper strand (the one whose preimage contains the crossing point with higher $Z$ coordinate) and a lower strand. This is recorded on the planar projection by removing a small neighbourhood of the crossing from the lower strand. The resulting finite set of arcs that rise and die near crossings is a \emph{link diagram}.\\
\begin{center}
\includegraphics*[scale=.5]{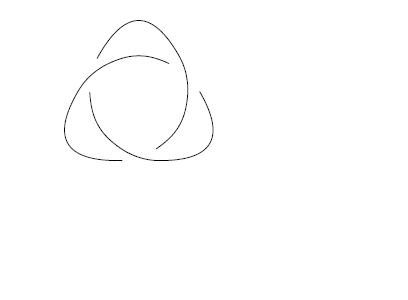}
\end{center}
\vspace*{-1cm}
\item The choice of an orientation induces an orientation on the arcs of the link diagram.\\
\begin{center}
\includegraphics*[scale=.5]{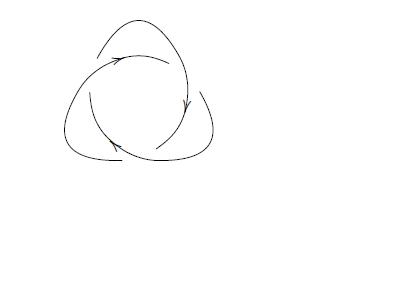}
\end{center}
\vspace*{-1cm}
\item Each arc produces a generator $a_j$ $(j=1,\dots,d)$. It represents the [homotopy class of the] loop starting from the base point, winding once as a left-handed screw around that arc and going back to the base point without wandering around anymore.\\
\begin{center}
\includegraphics*[scale=.5]{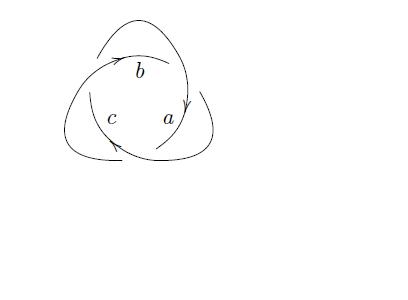}
\end{center}
\vspace*{-1cm}
\item Each crossing produces a relation $r_j$ $(j=1,\dots,d)$, according to the following rules:
\begin{equation*}
\xy 0;/r2pc/:
(0,0)*{\xy\knotholesize{8pt} \xunderh<<|>><<{b}|{x}>{a} \endxy};
(0,-1)*{xa=bx};
(4,0)*{\xy\knotholesize{8pt} \xoverh<>|<>><{a}|{x}>{b} \endxy};
(4,-1)*{ax=xb};
(6,0)*{};
\endxy
\end{equation*}
(by a simple combinatorial argument the number of arcs equals the number of crossings).\\
\begin{center}
\includegraphics*[scale=.5]{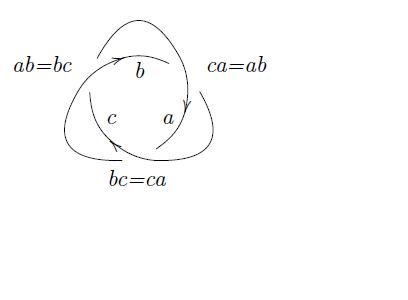}
\end{center}
\vspace*{-1cm}
\item The link group is then $G_L=\gen{a_1,\dots,a_d\mid r_1,\dots,r_d}$ (cf. \cite[Chapter VI]{crofox}).
\end{enumerate}

If in particular $L$ is a knot, there is always exactly one redundant relation. More precisely, knot groups have \emph{deficiency} $1$, that is, they admit a presentation with $(number\mbox{ }of\mbox{ } generators)= (number\mbox{ } of\mbox{ } relators)+1$ and no presentations with $(number\mbox{ }of\mbox{ } generators)= (number\mbox{ } of\mbox{ } relators)+k$ for $k>1$.
On the contrary, link groups can have arbitrary deficiency greater or equal to $1$. This is related to ``splitness''.
A link $L$ is \emph{split} if there exists a $2$-sphere $S^2\subseteq S^3\setminus L$ that separates $S^3$ into two $3$-balls, each containing at least one component of $L$. Such a $S^2$ is called a \emph{splitting $2$-sphere}. If this is the case, Seifert and Van Kampen's Theorem tells that the link group is the free product of the link groups of the two ``sublinks'' lying on the opposite sides of the $2$-sphere. Hence the deficiency of the group of the whole link is the sum of the deficiencies of the sublink groups. Then, by induction, we can construct link groups with arbitrarily large deficiency. For example, the link group of
\[
L_d=\bigsqcup_{i=1}^d \{(x,y,z)\in\R^3\mid x^2+y^2=1, z=i\}\subseteq\R^3\subseteq S^3
\]
is free on $d$ generators and has deficiency $d$.

\section[Homology of links]{Homological properties of link groups}
A \emph{tubular neighbourhood} $V$ of $L$ is a collection of open solid tori $V^{(i)}=L^{(i)}\times B^2$, each centered at one link component $L^{(i)}$, so thin as to be disjoint and non-self-intersecting (these conditions can be always fulfilled thanks to tameness). Then $S^3\setminus L$ is homeomorphic to $S^3\setminus\overline{V}$ and deformation-retracts onto $S^3\setminus V$ (from now on, we will denote $X_L$, or simply $X$ if there could be no confusion, the complement of a tubular neighbourhood $V$ of a link $L$ in $S^3$). So $G_L$ can be viewed as $\pi_1(X, x)$, where $x\in\partial X$. The boundary $T:=\partial X=\partial V$ is a disjoint union of classical tori $T^{(i)}=\partial V^{(i)}$.

Suppose that $L$ is a nontrivial knot. The inclusion-induced homomorphism of $H:=\pi_1(\partial X,x)\cong \Z^2$ into $G_L$ is injective. We can choose as generators of $H$ [the classes of] two simple closed curves $m,l$ such that (cf. \cite[Section 3.A]{burzie})
\begin{enumerate}
	\item[(a)] $m$ is null-homologous in $V$ but not in $\partial V$ (and hence not in $X$);
	\item[(b)] $l$ is homologous to $L$ in $V$ and to $0$ in $X$;
	\item[(c)] $m\cap l=\{x\}$;
	\item[(d)] the \emph{linking numbers} (integers that measure how many times a simple closed curve winds around another, see \cite{ricca} for the formal definition and interesting historical remarks) are $\mathrm{lk}(m,L)=1$ and $\mathrm{lk}(l,L)=0$.
\end{enumerate}
$m$ is referred to as a \emph{meridian}, $l$ as a \emph{longitude}. With a painless abuse, we will use the same terminology for the images of $m$ and $l$ in the knot group. Actually, thanks to the aforementioned injectivity, we shall think of the whole $H$ as a subgroup of $G_L$ and simply make no distinction at all between $m$ and $l$ and their respective images. Under that convention, $H$ is called a \emph{peripheral subgroup}. In particular, we can arrange [the image of] $m$ to coincide with a generator of the knot group.

Either using topological intuition or looking at the shape of the Wirtinger relations, one can easily see that all the generators of $G_L$ are conjugate to either $m$ or its inverse (for this reason, we will extend our abuse of language and call \emph{meridian} every generator of $G_L$). As a consequence, $H_1(X)=G_L^{ab}=G_L/[G_L,G_L]$ is infinite cyclic generated by the class of $m$ (cf.\cite[Theorem 2A.1]{hatcher} for the first equality). In other words, there is an isomorphism $W\colon G_L/[G_L,G_L]\to\Z$ sending $m[G_L,G_L]$ to $1$. It is the morphism induced on the quotient by $\mathrm{lk}(\slot,L)$ and will be called the \emph{winding number}, since it counts the number of times a simple closed curve winds around the knot.

In fact, using the Mayer-Vietoris sequence in homology (cf. \cite[Section 2.2]{hatcher}) associated to the pair $(X,\overline V)$, along with the knowledge of the homology groups of $S^3$ and $S^1\times S^1$ (cf. \cite[Section 2.1]{hatcher}), we are able to determine the homology of $X$ completely:
\[
H_0(X)\cong H_1(X)\cong\Z,\quad H_i(X)=0 \mbox{ for all } i\geq 2.
\]
As a consequence of Papakyriakopoulos's Sphere Theorem (\cite{papa}), $X$ is aspherical (i.e., its higher homotopy groups vanish). In other words, it is an Eilenberg-MacLane space $K(G_L,1)$ for the group $G_L$. This means that the homology (with integer coefficients) of the group $G_L$ coincides with the homology of $X$ (cf. \cite[Section II.4]{brown}).

Some of the preceding features extend to links, some others do not. Specifically, if $L$ is a generic link, each component identifies a conjugacy class of peripheral subgroups, each generated by a meridian and a longitude, but the latter curve is rarely null-homologous in $X$. As an example, consider the Hopf link: up to isotopy, the longitude of one component serves as the meridian of the other. The generators of the link group are partitioned into conjugacy classes. Two generators are conjugate if and only if they are meridians relative to the same link component. This means that $H_1(X)=G_L^{ab}$ is free abelian generated by the classes of the meridians of the components of $L$. The complete description of the homology of $X$ is
\[
H_0(X)\cong\Z,\;\: H_1(X)\cong\Z^d,\;\: H_2(X)\cong\Z^{d-1},\;\: H_i(X)=0 \mbox{ for all } i\geq 3.
\]

Unfortunately, link complements need not be aspherical: in fact, if a link $L$ is split, then obviously any splitting $2$-sphere cannot be homotoped to a point within $S^3\setminus L$. However, this is the only obstruction to asphericity (again by Papakyriakopoulos's Sphere Theorem), so that a link complement is aspherical if and only if the link is non-split, if and only if its link group has deficiency $1$. Again, if this is the case, then the homology of the link group coincides with the homology of the link complement.

\chapter[Algebraic number theory]{A crash course in algebraic number theory}\label{chp:ant}
General references for this chapter are \cite{lang_ANT} and \cite{neukirch}.

Algebraic number theory is the study of finite field extensions of $\Q$. Fix such an extention $K/\Q$; $K$ is called an \emph{algebraic number field}. The set of \emph{algebraic integers} of $K$
\[
\OO_K=\{a\in K\mid a\mbox{ is a root of a monic }f\in\Z[T]\}
\]
is a ring whose field of fractions is $K$ itself. This ring plays in $K$ the same role as $\Z$ in $\Q$.

In general, $\OO_K$ is not a principal ideal domain, yet it enjoys a good deal of the properties of $\Z$: in particular, it is noetherian and integrally closed and each of its nonzero prime ideals is maximal. In other words, it is a \emph{Dedekind domain}. As a consequence, even if $\OO_K$ may not be a unique factorisation domain, it does have the property of \emph{unique factorisation of ideals}:
every ideal of $\OO_K$ can be uniquely factored (up to the order of factors) into a product of finitely many prime ideals.

Taking inspiration from the behaviour of $\Q$ with respect to the primes of $\Z$, we define a \emph{fractional ideal} of $\OO_K$ to be a $O_K$-submodule $M$ of $K$ such that there exists a $c\in\OO_K\setminus\{0\}$ for which $cM\subseteq\OO_K$. Note that $cM$, and hence $M$, are finitely generated thanks to noetherianity of $\OO_K$.
Fractional ideals serve as multiplicative inverses of ``classical'' ideals of $\OO_K$. More precisely,
the set of nonzero fractional ideals of $\OO_K$ form a group under multiplication. It is called the \emph{ideal group} of $\OO_K$ and denoted $I_K$.
The \emph{principal fractional ideals} are the fractional ideals $k\OO_K$ generated by a single element $k\in K$. With the exclusion of $0\OO_K$, they form a subgroup $P_K\leq I_K$ and the quotient
\[
Cl_K=I_K/P_K
\]
is the \emph{ideal class group} of $K$. It measures the obstruction to $\OO_K$ being a principal ideal domain.

\section{Ramification of prime ideals}\label{sec:ramification of ideals}
Let $K$ be an algebraic number field. A prime ideal $Q\id\OO_K$ is said to \emph{lie above} a prime ideal $P\id\Z$ if $Q\cap\Z=P$ (notation: $Q|P$). For all prime ideals $P\id\Z$ there exists a prime ideal $Q\id\OO_K$ with $Q|P$, but such $Q$ may not be unique. This is related with how $P$ decomposes in $\OO_K$.

Indeed, it might very well happen that the \emph{lift} $P\OO_K$ of a prime ideal $P\id\Z$ does not stay prime in $\OO_K$. In general it has a factorization (unique up to permutations)
\[
P\OO_K=Q_1^{e_1}\cdots Q_r^{e_r}
\]
into finitely many prime ideals of $\OO_K$. The $Q_i$s are precisely the prime ideals of $\OO_K$ that lie above $P$.
The natural number $e_i$ is called the \emph{ramification index} of $Q_i$ over $P$. $P$ is \emph{unramified} in $K$ if, for all $i=1,\dots,r$, $e_i=1$ and the residue field extension $(\OO_K/Q_i)/(\Z/P)$ is separable, otherwise it is \emph{ramified}. We say that $P$ \emph{splits completely} in $K$ if $r=[K\colon\Q]$. If, on the contrary, $P$ does lift to a prime ideal in $K$, we call it \emph{inert}.

The above terminology is also extended to arbitrary extensions $K/F$ of algebraic number fields, where the same decomposition phenomena might happen.

\begin{ex}
Consider the degree $2$ extension $\Q(i)/\Q$. The associated extension of the rings of algebraic integers is $\Z[i]/\Z$. It is easy to see that odd primes of $\Z$ are congruent modulo $4$ to either $1$ or $3$. Fermat's Theorem on the sum of two squares (cf. \cite[Art. 182]{gauss}) states that an odd prime $p$ of $\Z$ is a sum of $2$ squares if and only if $p\equiv1\mbox{mod }4$. As a consequence,
\begin{itemize}
 \item the lift of $(2)\id\Z$ decomposes as $(2)\Z[i]=(1+i)^2$: the ideal $(2)$ ramifies in $\Q(i)$;
 \item the lift of a prime $p\equiv1\mbox{mod }4$ splits into $2$ prime ideals, $(p)\Z[i]=(a+ib)(a-ib)$: in othe words, $(p)$ splits completely in $\Q(i)$;
 \item the lift of a prime $p\equiv3\mbox{mod }4$ remains prime in $\Z[i]$, as it cannot be decomposed in $\Z$ and the only new chance of decomposition in $\Z[i]$ is ruled out by Fermat's Theorem: $(p)$ is inert in $\Q(i)$.
\end{itemize}
Obviously this is a sandbox example, since $\Z[i]$ is a principal ideal domain and prime ideals may be identified with prime elements. In real life, finding the decomposition of even a single prime in some algebraic number field extension might be painful.
\end{ex}

Ramification is a \emph{local} property, in the following sense.
The localization $A_P$ of a Dedekind domain $A$ at a nonzero prime ideal $P$ produces a ring in which the only nonzero prime (and maximal) ideal is $(A\setminus P)^{-1}P$, that is, a \emph{local ring}. Actually, a very special local ring: it is a local Dedekind domain and not a field, i.e., a \emph{discrete valuation ring}. In some sense, the localization process is a zoom lens that focuses only on what happens near $P$, disregarding what the domain is like elsewhere, that is near the other primes. It goes without saying that the behaviour of the ideals in a discrete valuation ring is far more understandable than in a general Dedekind domain, since there one has to take care of just one nonzero prime ideal and all the other nonzero ideals are powers of it.
For this reason, the nicest properties of ideals in Dedekind rings are those unaltered by the localization process, as they can be studied in the quiet environment of a localization. The ramification index is such a property (cf. \cite[Section I.7]{lang_ANT}).
\section{Valuations and completions}\label{sec:valuations}
The name \emph{discrete valuation ring} suggests that this algebraic structure should have an equivalent definition by completely different means.
In fact, this is true and allows us to give two alternative descriptions of prime ideals in rings of algebraic integers.

An \emph{ordered group} is an abelian group $A$ endowed with a total order compatible with the group operation $+$ (that is, $a\leq b\Rightarrow a+c\leq b+c$).
The corresponding \emph{ordered group closure} is $\overline A=A\sqcup\{\infty\}$, where the new symbol $\infty$ is subject to the rules
\[\begin{array}{l}
\infty+\infty=a+\infty=\infty+a=\infty\\
a<\infty
\end{array}
\]

A \emph{valuation} on a field $F$ is a map $\nu\colon F\to \overline{A}$, with values in an ordered group closure, satisfying
\begin{enumerate}
\item[(v1)] $\nu(x)=\infty \Leftrightarrow x=0$;
\item[(v2)] $\nu(xy)=\nu(x)+\nu(y)$;
\item[(v3)] $\nu(x+y)\geq\min\{\nu(x),\nu(y)\}$.
\end{enumerate}
The subgroup $\nu(F^*)\leq A$ is the \emph{value group} of $\nu$.
In particular, every field admits the \emph{trivial valuation}, defined by $\nu_{tr}(x)=0$ for all $x\in F\neq 0$, which from now on we exclude unless otherwise specified.

Two valuations $\nu,\mu\colon F\to\overline\R$ are equivalent if there is a real number $s>0$ such that $\mu=s\nu$. A valuation with values in $\overline R$ is \emph{discrete} if its value group is $\Z$ up to equivalence.

An \emph{absolute value} of a field $F$ is a map $|\slot|\colon F\to\R$ satisfying
\begin{enumerate}
\item[(av1)] $|x|\geq 0$ and $|x|=0 \Leftrightarrow x=0$;
\item[(av2)] $|xy|=|x||y|$;
\item[(av3)] $|x+y|\leq|x|+|y|$.
\end{enumerate}
In particular, every field admits the \emph{trivial absolute value}, defined by $|x|_{tr}=1$ for all $x\in F^*$, which from now on we exclude unless otherwise specified.

The other absolute values fall into two species: the \emph{nonarchimedean} (or \emph{ultrametric}) absolute values are those satisfying the stronger version of Axiom (av3)
\begin{enumerate}
\item[(av3')] $|x+y|\leq\max\{|x|,|y|\}$;
\end{enumerate}
the other absolute values are called \emph{archimedean}.

Every absolute value on $F$ turns it into a metric space via the distance $d(x,y)=|x-y|$.
Two absolute values $|\slot|_1$ and $|\slot|_2$ of $F$ are \emph{equivalent} if they induce the same topology on $F$. This happens if and only if there is a real number $t>0$ such that $|\slot|_2=|\slot|_1^t$. Obviously an archimedean absolute value cannot be equivalent to a nonarchimedean one.
A \emph{place} of an algebraic number field $K$ is a class of equivalent absolute values of $K$. We will usually simply identify a place with each of its representatives.

Every nonarchimedean absolute value $|\slot|$ of $K$ induces a discrete valuation
\[
\nu\colon K\to\overline\R,\quad\nu(x)=\begin{cases}-\log|x|&x\neq 0\\ \infty&x=0
\end{cases}
\]
Viceversa, every valuation $\nu\colon K\to\overline\R$ induces an absolute value
\[
|\slot|\colon K\to\R,\quad|x|=q^{-\nu(x)}
\]
for a fixed $q>1$.
Equivalent discrete valuations induce equivalent nonarchimedean absolute values, or, in other words, induce the same nonarchimedean place.

Moreover, for a nonarchimedean place $|\slot|$ of $K$ and the corresponding discrete valuation $\nu$, the subset
\[
\OO_\nu=\{x\in K\mid |x|\leq 1\}=\{x\in K\mid \nu(x)\geq 0\}
\]
is a discrete valuation ring with maximal ideal
\[
P_\nu=\{x\in\OO\mid |x|<1\}=\{x\in K\mid \nu(x)> 0\}.
\]
Conversely, a prime ideal $P\id \OO_K$ determines a discrete valuation by setting 
\[
\nu_P(x)= v_P\quad \mbox{ where }(x)=\prod_{P_i\id\OO_K\mbox{ prime}}P_i^{v_{P_i}}
\]
and this valuation can be turned into a nonarchimedean place.
The associated discrete valuation ring satisfies
\[%\begin{array}{}
(\OO_K)_P=\left\{x=\frac{a}{s}\in K\mid a\in\OO_K,s\in\OO_K\!\setminus\! P\right\}=\left\{x\hskip-0.5pt\in\hskip-0.5pt K\mid \nu_P(x)\geq 0\right\}=\OO_{\nu_P}
%\end{array}
\]
Summing up, there is a correspondence 
\begin{equation}\label{eq:valuations&primes}
\xymatrix{
\mbox{nonarchimedean places of }K\ar@{<~>}[d]\\
\mbox{equivalence classes of discrete valuations on }K\ar@{<~>}[d]\\
\mbox{localizations of }\OO_K\mbox{ at nonzero prime ideals}\ar@{<~>}[d]\\
\mbox{nonzero prime ideals of }\OO_K
}
\end{equation}
In view of which archimedean places can be thought as ``\mbox{primes at infinity}''.

A field with absolute value $(F,|\slot|)$ (and associated valuation $\nu$) is \emph{complete} if every sequence in $F$ that is Cauchy with respect to the metric induced by $|\slot|$ converges in $F$. The \emph{completion} $F_\nu$ of $(F,|\slot|)$ is the field obtained quotienting the ring of Cauchy sequences in $F$ by the maximal ideal of sequences converging to $0$. $F$ embeds into $F_\nu$ as the subfield of [classes of] constant sequences and $|\slot|$ admits a unique extension to $F_\nu$, defined by $|\overline{(a_n)_{n\in\N}}|=\lim_{n\to\infty}|a_n|$. The completion of a field with absolute value is essentially unique and is complete with respect to the aforementioned extension of the absolute value.

Let $K/F$ an algebraic extension of number fields. As explained in \cite[Sections II.4,II.6,II.7,II.8]{neukirch}, a valuation $\nu$ of $F$ can be extended to a valuation $\mu$ of $K$, once a particular $F$-embedding of $K$ into the algebraic closure of $F_\nu$ is given. We borrow the same notation $\mu|\nu$ used for prime ideals. Let $K^\mu$ denote the compositum (union) of the completions $(K_i)_{\mu}$ of the \emph{finite} subextensions $K_i/F$ of $K/F$.
In particular, if $K/F$ is finite, then $L^\mu=L_\mu=LK_\nu$ is the completion of $L$. The advantage in passing from the language of ideals to that of valuations lies exactly in the possibility of passing from an arbitrary extension $K/F$ of number fields to the various \emph{completion extensions} $K_\mu/F_\nu$, which are easier to understand.

In view of the correspondence \eqref{eq:valuations&primes} we can speak of \emph{ramified} and \emph{unramified} valuations by looking at the corresponding ideals. The ramification indices have a characterization in the language of valuations.
If $\nu$ is a nonarchimedean valuation on $F$ and we extend it to a valuation $\mu$ on $K$, then the \emph{ramification index} of the extension $\mu|\nu$ is
\[
e_\mu=[\mu(K^*)\colon\nu(F^*)]
\]
The prime ideal $Q_i\id\OO_K$ corresponding to $\mu$ lies above the prime ideal $P\id\OO_F$ corresponding to $\nu$ and the ramification index of $Q_i$ over $P$ coincides with the ramification index of $\mu$ over $\nu$.

In conclusion, if $K/F$ is a Galois extension of number fields, it makes sense to speak of the places of $F$ that ramify in $K$. With a little abuse, we will also use the expression \emph{archimedean place} to mean an equivalence class of archimedean valuations.

\section{Galois cohomology}
The \emph{(absolute) Galois group} of an algebraic number field $K$ is the Galois group $G_K$ of its separable closure. The \emph{(Galois) cohomology groups} of $K$ with coefficients in a $G_K$-module $A$ are
\[
H^i(K,A)=H^i(G_K,A).
\]
Given a valuation $\nu$ on $K$, one can view a finite $G_K$-module $A$ as a $G_{K_\nu}$-module and define the cohomology groups $H^i(K_\nu,A)$ (with the convention that $H^0(K_\nu,A)$ denotes the \emph{modified} cohomology group, cf.\cite[Section I.2]{NSW}).
There is a canonical homomorphism
\begin{equation}\label{eq:'nghu}
H^i(K,A)\to H^i(K_\nu,A),
\end{equation}
which, on the level of Galois groups, is given by the restriction map
\[
H^i(G_K,A)\to H^i(G_{K_\nu},A)
\]
(cf. \cite[Section II.6]{serre}).

The morphisms \eqref{eq:'nghu} for varying $\nu$ can be bundled together in a homomorphism
\begin{equation}\label{eq:'nghu piu grosso}
H^i(K,A)\to \prod_{\nu} H^i(K_\nu,A).
\end{equation}

Let now $K/F$ be a Galois extension of algebraic number fields and let $S$ be a finite set of places of $F$ containing all the places that ramify in $K$ and the subset $S_\infty$ of all the archimedean places.
We denote $G_S$ the Galois group of the maximal Galois extension unramified outside $S$.
Let $A$ be a finite $G_S$-module whose order is a $S$-unit, that is, $\nu(|A|)=0$ for all $\nu\notin S$. Letting $\OO_S^*=\{x\in F^*\mid \nu(x)=0 \mbox{ for }\nu\notin S\}$, the \emph{dual} module of $A$ is $A'=\hom(A,\OO_S^*)$.

Let $E_{nr}$ denote the maximal unramified extension of the field $E$. For $\nu\notin S$, $Gal(\overline{F_\nu}/(F_\nu)_{nr})$ acts trivially on $A$, hence $A$ can be considered a $Gal((F_\nu)_{nr}/F_\nu)$-module, hence the groups
\[
H^i_{nr}(F_\nu,A)=H^i((F_\nu)_{nr}/F_\nu,A)
\]
are defined (cf.\cite[\S\ II.5.5]{serre}).
Let 
\[
P^i(F_S,A)=\left\{(c_\nu)\in\prod_{\nu}H^i(F_\nu,A)\mid c_\nu\in H^i_{nr}(F_\nu,A)\mbox{ for almost all }\nu\right\}.
\]
The morphism \eqref{eq:'nghu piu grosso} maps $H^i(F,A)$ to $P^i(F_S,A)$ (cf. \cite[Proposition 8.6.1]{NSW}); the kernel of this map is the \emph{$i$-th Shafarevi\v{c} group}. Moreover, there is a $9$-term exact sequence
\begin{equation}\label{eq:poitou-tate ANT}
\xymatrix{
0\ar[r] & H^0(F_S,A)\ar[r] & P^0(F_S,A)\ar[r] & H^2(F_S,A')^\vee\ar[dll]&\\
& H^1(F_S,A)\ar[r] & P^1(F_S,A)\ar[r] & H^1(F_S,A')^\vee\ar[dll]&\\
& H^2(F_S,A)\ar[r] & P^2(F_S,A)\ar[r] & H^0(F_S,A')\ar[r]^\vee&0\\
}
\end{equation}
called Poitou-Tate exact sequence (cf.\cite[(8.6.10)]{NSW}).
Here, for a torsion abelian group $M$, $M^\vee=\hom_{cont}(M,\Q/\Z)$ denotes the \emph{Pontryagin dual} of $M$ (cf. \cite[Section I.1]{serre}).
The groups $P^i(F_S,A)$ encode the local information about the field $F$, so Poitou-Tate sequence connects the local and the global behaviour of an algebraic number field. In the same spirit, Shafarevi\v{c} groups measure the amount of information lost passing from the global to the local point of view.

\chapter[Classifying spaces]{Classifying spaces of knot groups}\label{chp:classSpaces}
Let $K\subseteq S^3$ be a knot, with knot group $G$ and a tubular neighbourhood $V$, and set $X=S^3\setminus V$.
The space $X$ is connected, locally path-connected and semilocally simply connected, hence it admits a universal cover $\widetilde X$.

These two spaces have a rich topological structure: $X$ is a CW-complex whose higher homotopy groups are trivial as a consequence of Papakyriakopoulos's Sphere Theorem (cf. \cite{papa}). It follows that $\widetilde X$ is a free, contractible $G$-CW-complex. Moreover, the covering map $\pi\colon\widetilde X\to X$ is the projection to the quotient of the $G$-action and it coincides with the universal principal $G$-bundle $EG\to BG$.

Now, $X$ is obtained by removing an open solid torus from $S^3$: what modifications would affect the picture if one tried to glue $V$ back? That is, more precisely:
\begin{question}\label{qu:extend class sp}
Is it possible to find another contractible $G$-CW-complex $Z$ that includes $\widetilde X$ as a $G$-CW-subcomplex and such that the restriction of the projection $\varpi:Z\stackrel{\cdot/G}{\to}S^3$ to $\widetilde X$ coincides with $\pi$? And if so, what topological properties would the nicest such $Z$ have?
\end{question}

The $G$-action on $Z$ cannot be free, since each meridian of $G$ fixes some points in the preimage of $K$. Thus the best one can achieve is that $Z$ be a contractible $G$-CW-complex such that
\begin{enumerate}
\item $\varpi$ behaves like a covering map except over $K$; in a neighbourhood of each point of $K$, it acts like an analogue of an integer power map in a neighbourhood of $0\in\C$, but with ``infinite exponent'';
\item only [the cyclic subgroups generated by each of] the meridians fix some points;
\item these nonempty fixed-point sets are contractible.
\end{enumerate}
The first requirement is formalised in the concept of \emph{branched} (or \emph{ramified}) \emph{covering}, that we state here in the finite case only, since it is the case we will always deal with in the future.
\begin{defn}
Let $M$ and $N$ be $3$-manifolds and $f\colon M\to N$ a continuous surjection. Define $B_M=\{x\in M\mid f\mbox{ is not a homeomorphism in a}$ $\mbox{neighbourhood of }x\}$ and $B_N=f(B_M)$. Finally, view $D^2=\{z\in\C\mid|z|\leq 1\}$ and let, for a positive integer $e$, $pw_e\colon D^2\to D^2$, $pw_e(z)=z^e$. Then $f$ is a \emph{finite branched covering map} branched over $B_N$ if
\begin{itemize}
\item[(a)] $f|_{M\setminus B_M}\colon M\setminus B_M\to N\setminus B_N$ is a covering;
\item[(b)] for any $x\in B_M$, there are (closed) neighbourhoods $U$ of $x$ in $M$ and $W$ of $f(x)$ in $N$ and homeomorphisms $\alpha\colon U\to D^2\times [0,1]$ and $\beta\colon W\to D^2\times [0,1]$ such that for some positive integer $e$
\[
\xymatrix{
U\ar^-{f}[rr]\ar^-{\alpha}[d] && V\ar^-{\beta}[d]\\
D^2\times [0,1]\ar^{pw_e\times id_{[0,1]}}[rr] && D^2\times [0,1]
}
\]
is a commutative diagram. The number $e$ is the \emph{ramification index} at $x$.
\end{itemize}
\end{defn}
%\begin{rem}
%
%\end{rem}
The last two requirements are encoded in the concept of \emph{classifying space for a family of subgroups}.

\section{Classifying spaces}
\begin{defn}\label{def:family}
Let $G$ be a group. A \emph{family of subgroups} (or simply \emph{family}) of $G$ is a set of subgroups of $G$ closed under conjugation and taking subgroups.

The family \emph{generated} by a subgroup $H$ consists of all conjugates of $H$ and all their subgroups and will be denoted $\F(H)$. When we will be concerned with families generated by infinite cyclic groups, we will briefly call $\F(\gen{a_1},\dots,\gen{a_n})$ the family of $a_1,\dots,a_n$.
\end{defn}

\begin{defn}
Let $\F$ be a family of the group $G$. A \emph{[model for the] classifying space} $E_\F(G)$ \emph{of the family} $\F$ is a $G$-CW-complex $X$ such that for all subgroups $H\leq G$ the fixed point set 
\[
 X^H=\left\{ x\in X\mid\forall h\in H:h(x)=x\right\}
\]
is contractible if $H\in\F$ and empty otherwise.
\end{defn}
The space $Z$ we are looking for in Question \ref{qu:extend class sp} is nothing but the classifying space of a knot group for the family of the meridians.

\begin{rem}
There is a more abstract equivalent definition of classifying space of $\F$ as a terminal object in the $G$-homotopy category of $G$-CW-complexes whose isotropy groups belong to $\F$ (cf. \cite{luecksurvey}). This is interesting in that it automatically guarantees the uniqueness of models up to $G$-homotopy, but it is useless for \emph{finding} models.

As regards existence, it is guaranteed for all families of all groups by Milnor's construction (cf. \cite{milnor1956}). Unfortunately, that construction often gives a too big space to deal with.
In fact, usually one is interested not only in the mere existence of of such spaces, but in performing effectively some algebraic topology on them. A fundamental issue, then, is to find models that are as simple and concrete as possible, in order
to carry explicit computations.
\end{rem}

This issue of course affects also our situation: now that we know the first part of Question \ref{qu:extend class sp} has a positive answer, it is time to address the second part. As said, lots of information are encoded in the total space $EG$ of the universal principal bundle of the knot group $G$ (that is, the universal cover $\widetilde X$). On the other hand, $EG$ can be seen as the classifying space of the trivial family $\Big\{\{1\}\Big\}$, which the family of meridians is a superset of. Is it then possible to use $EG$ as a kind of foundations which to construct the classifying space $Z$ on? This is a particular instance of the problem of building classifying spaces for larger families from those for smaller ones.

A result by L\"uck and Weiermann (\cite{luewei2007}) does the trick.
Let $\F\subseteq\G$ be two families of the group $G$. An equivalence relation $\sim$ on $\G\setminus\F$ with the additional properties
\begin{eqnarray}
H,K\in \G\setminus\F, H\subseteq K \Longrightarrow H\sim K\label{tilde1}\\
\label{tilde2} H,K\in \G\setminus\F, H\sim K \Longrightarrow \forall g\in G: g^{-1}Hg\sim g^{-1}Kg
\end{eqnarray}
will be called a \emph{strong equivalence relation}.
Given such $\sim$, denote $[G\setminus H]$ the set of $\sim$-equivalence classes and $[H]$ the $\sim$-equivalence class of $H$ and define the subgroup of $G$
$$
N_G[H]=\{g\in G\mid [g^{-1}Hg]=[H]\}.
$$
Property (\ref{tilde2}) says that the conjugation of subgroups induces an action on $[G\setminus H]$ with respect to which $N_G[H]$ is the stabiliser of $[H]$.
Finally construct the family of subgroups of $N_G[H]$
$$
\G[H]=\{K\leq N_G[H]\mid K\in\G\setminus \F, [K]=[H]\}\cup (\F\cap N_G[H]),
$$
where $\F\cap N_G[H]$ is a shorthand for $\{K\leq N_G[H]\mid K\in\F\}$.
\begin{thm}[L\"uck \& Weiermann]\label{thm:lueckweiermann} Let $I$ be a complete system of representatives $[H]$ of the $G$-orbits in $[\G\setminus \F]$ under the $G$-action coming from conjugation. Choose arbitrary $N_G[H]$-CW-models for $E_{\F\cap N_G[H]}(N_G[H])$ and $E_{\G[H]}(N_G[H])$,
and an arbitrary $G$-CW-model for $E_\F(G)$. Define a $G$-CW-complex $Y$ by the
cellular $G$-pushout
\begin{equation}\label{eq:pushout diagram}
\xymatrix{
\bigsqcup_{[H]\in I}G\times_{N_G[H]}E_{\F\cap N_G[H]}(N_G[H])\ar^-\phi[rrr]\ar^{\bigsqcup_{[H]\in I}id_G\times\psi_{[H]}}[d]&&& E_\F G\ar[d]\\
\bigsqcup_{[H]\in I}G\times_{N_G[H]}E_{\G[H]}(N_G[H])\ar[rrr]&&&Y
}
\end{equation}
such that the $G$-map $\phi$ is cellular and all the $N_G[H]$-maps $\psi_{[H]}$ are inclusions or viceversa.

Then $Y$ is a model for $E_\G G$.
\end{thm}

\section[Classifying spaces for meridians]{Classifying spaces for the meridians of a knot}
We wish to apply L\"uck and Weiermann's Theorem in this setting: $G=G_L$ the group of the knot $L$, $\F$ the trivial family, $\G$ the family generated by a generator of $G$.

Now, if the knot $L$ is trivial, then $\G$ is the set of all subgroups of $G\simeq\Z$, so $E_\G G$ is a point. This degenerate situation having been settled, in what follows we assume $L$ to be nontrivial.

\subsection{Structure of $\G$}
Let $a_1,\dots a_d$ be the generators of $G$ and select one among them, say $a:=a_1$. The others are conjugate to either $a$ or $a^{-1}$, as a consequence of the shape of the Wirtinger relators (cf. Chapter \ref{chp:knotTheory}). This means that $\G$ is the family generated by $a$, according to Definition \ref{def:family}. We will refer to $a$ as the \emph{chosen meridian}.
Thus the subgroups in the family have a nice explicit shape:
$$
\G=\{\gen{g^{-1}a^ig}\mid g\in G, i\in\Z\}.
$$
\begin{rem}\label{rem:GG partialy ordered}
In particular, $\G$ is a partially ordered set with respect to inclusion, it has maximal elements $\gen{g^{-1}ag}, g\in G$ and each element of $\G$ lies in one of those.% These properties will play a key role in the following.
\end{rem}

\subsection{Definition of $\sim$}
If we fix a $H\in\G\setminus\F$ and we set $\uparrow\! H:=\{K\in\G\setminus\F\mid K\geq H\}$ (the \emph{up-set} of $H$), $\downarrow\! H:=\{K\in\G\setminus\F\mid K\leq H\}$ (the \emph{down-set} of $H$) and $\updownarrow\! H:=\uparrow\! H\cup\downarrow\! H$, then Property \eqref{tilde1} of strong equivalence relations says that $\updownarrow\! H\subseteq [H]$. Clearly, an equivalence relation for which the reverse inclusion holds has very little hopes to exist in general and no hopes at all in our setting, as for example $\gen{a^{2i}}\in [\gen{a^{3i}}]\setminus\updownarrow\!\gen{a^{3i}}$. On the other hand, it is natural to ask the equivalence relation to be as fine as possible.

Consider the relation
$$
H\sim K \Longleftrightarrow \exists L\in\G\setminus\F :L\leq H,L\leq K;
$$
it is clearly reflexive and symmetric.
\begin{lem}
The following statements are equivalent:
\begin{description}
\item[(i)] $\sim$ is transitive;
\item[(ii)] $\forall H\in\G\setminus\F :\, \downarrow\! H$ is a downward-directed set with respect to inclusion;
\item[(iii)] $\forall\mbox{ maximal }H\in\G\setminus\F :\, \downarrow\! H$ is a downward-directed set with respect to inclusion.
\end{description}
(A partially ordered set $(P,\preceq)$ is downward-directed if for all $x,y\in P$ there is a $z\in P$ such that $z\preceq x$ and $z\preceq y$).
\end{lem}
\begin{proof}\
\begin{description}
\item[(i)$\Rightarrow$(ii)]
If $K_1, K_2\in\downarrow\! H$, then by definition $K_1\sim H$ and $H\sim K_2$, so that by transitivity $K_1\sim K_2$, which means $\exists L\in\G\setminus\F :L\leq K_1,L\leq K_2$ (obviously such a $L$ is in $\downarrow\! H$ by transitivity of $\leq$).
\item[(ii)$\Rightarrow$(i)] Let $H_1\sim H_2$ and $H_2\sim H_3$, that is $\exists K_1\in\G\setminus\F :K_1\leq H_1,K_1\leq H_2$ and $\exists K_2\in\G\setminus\F :K_2\leq H_2,K_2\leq H_3$. In particular $K_1, K_2\in\downarrow\! H_2$, so that by directedness there exists a $L\in\downarrow\! H_2\subseteq\G\setminus\F$ which is included in both. All the three $H_i$s include this $L$, hence $H_1\sim H_3$.
\[
\xymatrix{
H_1\ar@{-}[dr] && H_2\ar@{-}[dr]\ar@{-}[dl] && H_3\ar@{-}[dl]\\
& K_1\ar@{-}[dr] && K_2\ar@{-}[dl] & \\
&& L &&
}
\]
\item[(ii)$\Rightarrow$(iii)] Obvious.
\item[(iii)$\Rightarrow$(ii)] Each subgroup in $\G\setminus\F$ is included in a maximal one, and downward-directedness is inherited by subsets.
\end{description}
\end{proof}

As all the elements of $\G\setminus\F$ are infinite cyclic, $\sim$ is an equivalence relation via Condition (ii). It also satisfies the two additional axioms of strong equivalence relations.
Recalling Remark \ref{rem:GG partialy ordered}, we can find a maximal element in each $\sim$-equivalence class and we choose it as the representative of its class.

\subsection{$N_G[H]$ and $\G[H]$}
As far as we are concerned with knots, the conjugation action of $G$ on $[\G\setminus\F]$ is transitive by the very definition. So the stabilisers are all conjugate to one another and we only need to study $N_G[\gen{a}]$.

\begin{lem}\label{lem:non-cable normalizer}
If $L$ is a nontrivial knot, then $N_G[\gen{a}]=C_G(a)$.
\end{lem}
\begin{proof}
By the definition of $\sim$, a certain $t\in G$ lies in $N_G[\gen{a}]$ if and only if $\gen{t^{-1}at}\cap \gen{a}\neq\{1\}$, if and only if $\exists i,j\in\Z\setminus\{0\}$ such that $t^{-1}a^it=a^j$. If we project the last equality onto $G/[G,G]$, which is infinite cyclic generated by the image of $a$, we see that $i=j$. Hence $t$ centralises a power of $a$. But, according to \cite[Corollary 3.7]{jacshaperi},
the centraliser in $G$ of an element of the peripheral subgroup coincides with the centraliser of any power of the given element\footnote{I thank Iain Moffatt for pointing my nose on that reference.}.

Summing up, all $t\in N_G[\gen{a}]$ centralise $a$ and the viceversa is obvious.
\end{proof}
\begin{rem}
This also means that Remark \ref{rem:GG partialy ordered} can be strengthened: each element of $\G$ is included in a \emph{unique} maximal element.
\end{rem}

The same idea yields:
\begin{lem}\label{lem:non-cable family}
if $L$ is a nontrivial knot, then $\G[\gen{a}]=\all\gen{a}:=\,\downarrow\!\gen{a}\cup\big\{\{1\}\big\}$.
\end{lem}
\begin{proof}
One has just to observe that the simultaneous conditions $H\in\G\setminus\F$ and $[H]=[\gen{a}]$ 
mean $H=\gen{t^{-1}a^kt}$, for some $a\in\N^*$ and some $t\in G$ which, by the same argument as in the previous lemma, lies in $C_G(a)$. It follows that $H=\gen{a^k}\leq \gen{a}$.
\end{proof}

\subsection{Prime knots}
There is a binary operation, called \emph{knot sum}, that turns the set of oriented knot types in $S^3$ into a commutative monoid.
It is defined by the following procedure.
Consider disjoint projections of each knot on a plane. Find a rectangle in the plane a pair of opposite sides of which are arcs along each knot but is otherwise disjoint from the knots and such that there is an orientation of the boundary of the rectangle which is compatible with the orientations of the knots. Finally,
join the two knots together by deleting these arcs from the knots and adding the arcs that form the other pair of sides of the rectangle.
The resulting connected sum knot inherits an orientation consistent with the orientations of the two original knots, and the oriented ambient isotopy class of the result is well-defined, depending only on the oriented ambient isotopy classes of the original two knots.

A knot is \emph{prime} if it is nontrivial and it cannot be obtained as the knot sum of nontrivial knots.

When a knot is prime, it follows from Simon's results \cite{simon1976} that the centraliser of a meridian is the peripheral subgroup containing it, in symbols $C_G(a)=\gen{a,l}=H\simeq\Z^2$, where $l$ is a longitude corresponding to $a$. So Diagram (\ref{eq:pushout diagram}) becomes
\begin{equation}\label{eq:prime non-cable pushout diagram}
\xymatrix{
G\times_{\gen{a,l}}E\Z^2\ar^\phi[rrr]\ar^{id_G\times\psi}[d] &&& EG\ar[d]\\
G\times_{\gen{a,l}}E_{\all\gen{a}}\Z^2\ar[rrr]&&&E_\G G.
}
\end{equation}
The top-left space is a disjoint union of copies of $\R^2$ indexed by $G/\gen{a,l}$, glued via $\phi$ to the boundary components of the top-right space. This one is well-known to be a $3$-manifold, with boundary a disjoint union of planes indexed exactly by $G/\gen{a,l}$ (this way $\phi$ can clearly be chosen as to be a cellular map). As for the bottom-left corner, consider $\R^2$, on which $\gen{a,l}$ acts freely by integer translations, as a model for $E(\Z^2)$, and $\R$, on which $l$ acts by integer translations while $a$ fixes every point, as a model for $E(\Z^2)^{\gen{a}}$. Then the space in that corner can be chosen to be the disjoint union, again indexed by $G/\gen{a,l}$, of copies of the mapping cylinder $Cyl(\pi\colon\R^2\to\R)$ of the projection $E(\Z^2)\to E(\Z^2)^{\gen{a}}\simeq E(\Z^2)/\gen{a}$.
This choice guarantees the map $\psi$ to be the inclusion of $\R^2$ into the mapping cylinder, as required in L\"uck and Weiermann's Theorem \ref{thm:lueckweiermann}.
Finally, $E_\G G$ is the $3$-dimensional $G$-CW-complex obtained by gluing each $Cyl(\pi\colon\R^2\to\R)$, along its $\R^2$ copy, to one boundary component of $E(G)$.

Summing up, for a prime knot the classifying space $E_\G G$ is given by the pushout
\begin{equation}\label{eq:prime diagram explicit}
\xymatrix{
\bigsqcup_{G/H}\R^2\ar^\phi[rrr]\ar^{id_G\times\psi}[d] &&& EG\ar[d]\\
\bigsqcup_{G/H}Cyl(\pi\colon\R^2\to\R)\ar[rrr]&&&E_\G G.
}
\end{equation}

\subsection{Non-prime knots}\label{ssec:non-prime knots}
As Schubert showed in its doctoral thesis \cite{schubert1949}, any non-prime knot $L$ admits an essentially unique decomposition into a knot sum of prime factors $L=K_1\sharp\dots\sharp K_r$. According to Noga (\cite{noga1967}), the centralizer $C_{G'}(a)$ of the chosen meridian in the commutator subgroup of $G$ is free of rank equal to the number $r$ of prime factors in the above decomposition, and it is generated by the longitudes $l_1,\dots,l_r$ of those factors.

But then,
\begin{equation*}
\frac{C_G(a)}{C_{G'}(a)}=\frac{C_G(a)}{C_G(a)\cap G'}\simeq \frac{C_G(a)G'}{G'}\leq \frac{G}{G'}.
\end{equation*}
Since the last group is infinite cyclic, the first one is so as well. Moreover, since $a\in C_G(a)\setminus G'$, a generator of this group is $aC_{G'}(a)$.
This amounts to saying that there is a right-split exact sequence
\begin{equation*}
1\rightarrow C_{G'}(a)\rightarrow C_{G}(a)\rightleftarrows \gen{a}\rightarrow 1
\end{equation*}
or equivalently, since $a$ is central in $C_{G'}(a)$,
\begin{equation*}\label{eq:centralizer non-prime non-cable}
C_{G}(a)=\gen{a}\times C_{G'}(a)\simeq \Z\times F_r
\end{equation*}
(where $F_r$ denotes the free group of rank $r$).

In this setting, Diagram \eqref{eq:pushout diagram} is
\begin{equation}\label{eq:non-prime non-cable pushout diagram}
\xymatrix{
G\times_{\gen{a,l_1,\dots,l_r}}E(\Z\times F_r)\ar^\phi[rrr]\ar^{id_G\times\psi}[d] &&& EG\ar[d]\\
G\times_{\gen{a,l_1,\dots,l_r}}E_{\all\gen{a}}(\Z\times F_r)\ar[rrr]&&&E_\G G.
}
\end{equation}
Setting $\mathcal{C}(F_r)$ to be (the geometric realization of) the Cayley graph of $F_r$ (with respect to the generating set made of the generators of $F_r$ and their inverses), we can choose, as a model for the top-left corner, the disjoint union of copies of $\R\times \mathcal{C}(F_r)$ indexed by $G/\gen{a,l_1,\dots,l_n}$ and, as a model for the bottom-left corner, the mapping cylinder of the projection $\R\times \mathcal{C}(F_r)\rightarrow \mathcal{C}(F_r)$.

\section{The Hopf link}
With the same machinery we can also build a classifying space for the family of the meridans of the group $G=\gen{a,b\mid ab=ba}\simeq\Z^2$ of the Hopf link.
\[\knotholesize{4mm}
\xygraph{
!{0;/r1.5pc/:}
!{\vunder}
!{\vunder-}
[uur]!{\hcap[2]<><{b}}
[l]!{\hcap[-2]<<<{a}}
}
\]

In this link, the meridian of each component serves as the longitude of the other. Hence the classifying space is the double mapping cylinder $Cyl(\R\leftarrow \R^2\rightarrow\R)$ of the projections onto the two components of $EG=\R^2$. Here the copy of $\R$ which is the image of the first projection carries the translation action of $b$ and the trivial action of $a$ (i.e., it is $EG^{\gen a}$), and conversely for the other copy (which is then $EG^{\gen b})$.

There are two reasons that make this example interesting. The first is that it suggests a track of further research, namely the construction of classifying spaces for the meridians of \emph{links}. The second is that this is almost certainly the only case in which a classifying space of ours comes out well in a photograph.
\begin{center}
 \includegraphics[scale=.5]{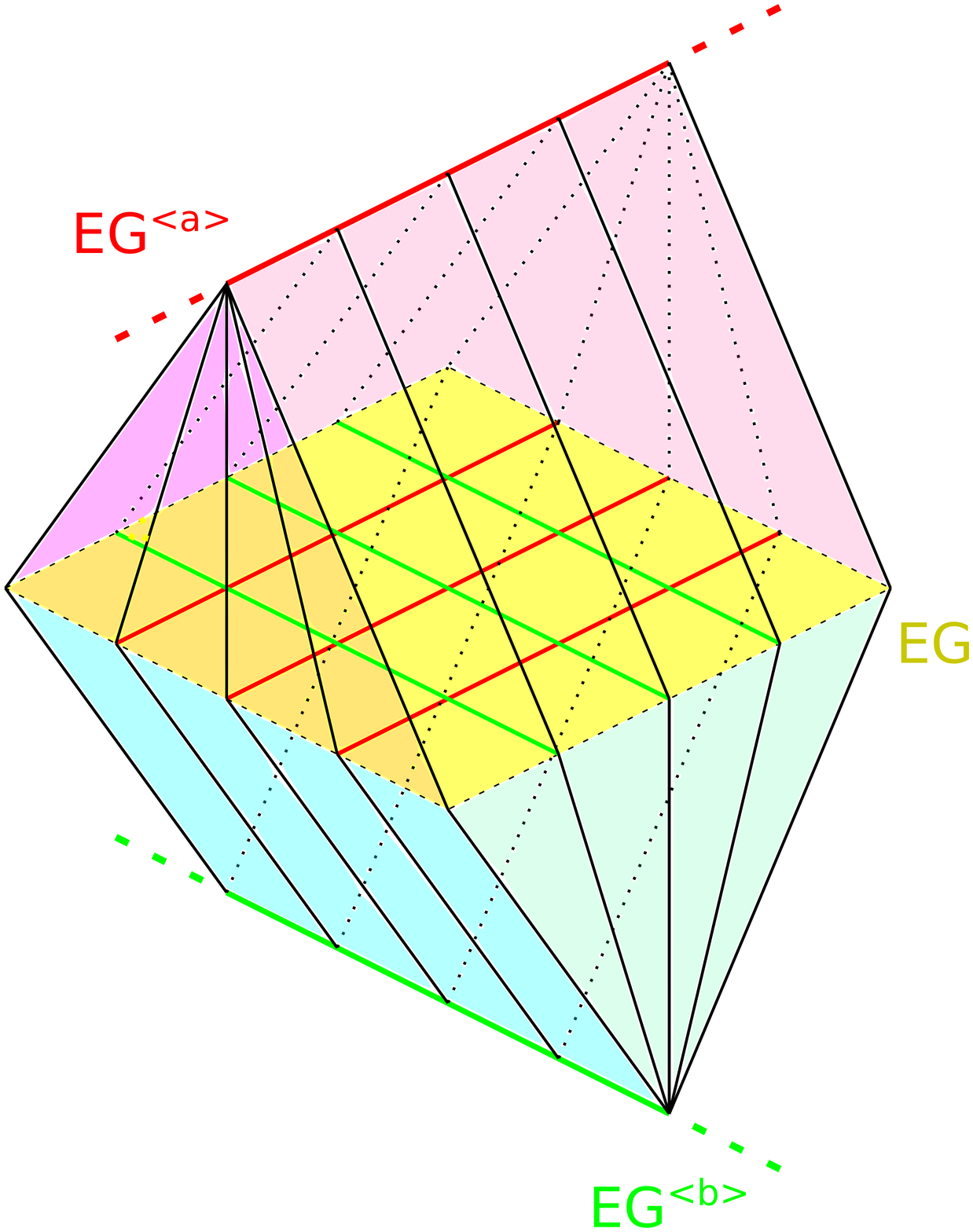}
\end{center}

\chapter[Coverings and field extensions]{Branched coverings of $S^3$ and extensions of number fields}
\section{A motivating analogy}
We recall a property of prime ideals in Dedekind rings (cf. e.g. \cite[\s~1.7]{lang_ANT}). It generalises and completes what we said in Section \ref{sec:ramification of ideals}.
Let $A$ be a Dedekind ring, $K$ its quotient field, $L/K$ a finite separable extension and $B$ the integral closure of $A$ in $L$. If $P$ is a prime ideal of $A$, then $PB$ is an ideal of $B$ and has a factorization (unique up to permutations)
\[
PB=Q_1^{e_1}\cdots Q_r^{e_r}
\]
into finitely many prime ideals of $B$. The $Q_i$s are precisely the prime ideals of $B$ that lie above $P$ meaning that $Q_i\cap A=P$ (notation: $Q_i|P$).

The natural number $e_i$ is called the \emph{ramification index} of $Q_i$ over $P$. Moreover, for each $i$, let $f_i$ be the (finite) degree of $B/Q_i$ as a field extension of $A/P$ (recall that nonzero prime ideals of Dedekind rings are maximal, hence $B/Q_i$ and $A/P$ are fields).

One can show that
\[
[L:K]=\sum_{i=1}^re_if_i.
\]
If $L/K$ is Galois, then all the $Q_i$s are conjugate to each other, hence all the ramification indices (resp. the residue class degrees) are equal to the same number $e$ (resp. $f$). Thus, in particular, the preceding formula becomes
\begin{equation}\label{eq:efr=degree}
[L:K]=efr.
\end{equation}

A similar formula holds for finite regular coverings of $S^3\setminus V$, where $V$ is a tubular neighbourhood of a knot $K\subseteq S^3$ (cf. \cite[Chapter 5]{morishita}).
In fact, let $G$ be the knot group of $K$ and $H=\gen{a,l}\leq G$ be the peripheral subgroup containing our favourite meridian $a$. Consider a normal subgroup $U\trianglelefteq G$ of finite index. Then $UH$ is a subgroup of $G$ and
\begin{equation}\label{eq:efr for knots}
\begin{split}
|G\colon U|&=|G\colon UH|\cdot|UH\colon U|=|G\colon UH|\cdot|H\colon H\cap U|=\\
&=|G\colon UH|\cdot|H\colon\gen{a}(H\cap U)|\cdot|\gen{a}(H\cap U)\colon H\cap U|\\
&=|G\colon UH|\cdot|H\colon\gen{a}(H\cap U)|\cdot|\gen{a}\colon \gen{a}\cap U|
\end{split}
\end{equation}
In topological terms, $U$ defines a finite regular covering $\pi_U\colon EG/U\to BG$ of the knot exterior $BG$. The $G$-action on $EG$ induces a transitive permutation of the connected components of $\partial EG$, which are planes, and the peripheral subgroup $H$ is the (setwise) stabiliser of one such component $P_0$. Hence the subgroup $UH$ is the (setwise) stabiliser of the connected component $T_0=\pi_U(P_0)$ (a torus) of $\partial EG/U$. It follows that the cosets in $G/UH$ parametrise the connected components of the fibre $\pi_U^{-1}(\partial BU)$. So $|G\colon UH|$ is analogous to $r$ in Equation \eqref{eq:efr=degree}. In the same spirit, $|\gen{a}\colon \gen{a}\cap U|$ is the least positive integer $k$ such that $a^k\in U$, hence it is analogous to $e$ in the same equation. Finally, $|H\colon\gen{a}(H\cap U)|$ is analogous to $f$.

Actually, the above reasoning also holds for $K$ a link with components $K_1,\dots,K_c$, provided the peripheral subgroups of each component are nondegenerate, i.e., isomorphic to $\Z^2$. Such a link will be called \emph{peripherally nontrivial}.
For $i=1,\dots,c$, let $V_i$ be a tubular neighbourhood of $K_i$, disjoint from the others, and let $H_i=\gen{a_i,l_i}$ be the peripheral subgroup of $G$ generated by a meridian $a_i$ and the corresponding longitude $l_i$ of the component $K_i$. $G$ permutes the connected components of $\partial EG/U$ in such a way that two of them are in the same orbit if and only if they are relative to the same component $K_i$ of $K$, that is, if and only if they project to $\partial V_i$.% In this case, these connected components are in $1-1$ correspondence with those of $\pi_U^{-1}(K_i)$. 
The stabiliser of a connected component of $\partial EG/U$ relative to $K_i$ is $UH_i$. Thus, the connected components of the fibre $\pi^{-1}(\partial V_i)$ are indexed by $G/UH_i$ and it still holds that
\[
\xymatrix@C-8pt@R-5pt{
[G\colon U]&=&[G\colon UH_i]\ar@{~>}[d]&[H_i\colon\gen{a_i}(H_i\cap U)]\ar@{~>}[d]&[\gen{a_i}\colon \gen{a_i}\cap U]\ar@{~>}[d]\\
&& \vspace{8pt} r_i&f_i & e_i.
}
\]
In other words, the components $K_1,\dots,K_c$ are the analogues of the distinct primes in the ground field of an algebraic number field extension.

It is possible to push this comparison a little forward: in number theory the following holds.
\begin{prop}[\cite{neukirch}, Chapter VI, Corollary 3.8]\label{prop:cor3.8neukirch}%[
Let $L/F$ be a finite extension of algebraic number fields. If almost all primes of $F$ split completely in $L$, then $L=F$.
\end{prop}

The next result can be seen as a partial analogue in knot theory.
\begin{prop}\label{prop:analogue of cor3.8}
Let $G$ be the group of a peripherally nontrivial link $K$ with components $K_1,...,K_c$ and relative tubular neighbourhoods $V_1,\dots,V_c$. Let $U\id G$ a normal subgroup of finite index and $\pi_U\colon EG/U\to BG$ the finite covering induced by $U$. Then $G=U$ if and only if $\forall i=1,\dots,c:|\pi_0(\pi^{-1}(\partial V_i))|=|G:U|$.
\end{prop}
\begin{proof}
For $i=1,\dots,c$, let $H_i=\gen{a_i,l_i}$ be the peripheral subgroup of $G$ generated by the meridian $a_i$ and the longitude $l_i$ of the component $K_i$. It has already been said that $|\pi_0(\pi^{-1}(\partial V_i))|=|G:UH_i|$. On the other hand, $|G:U|=|G:UH_i|\cdot|UH_i:U|$. Hence the hypothesis 
\begin{equation}\label{eq:blaaaaah}
 [G:U]=|\pi_0(\pi^{-1}(\partial V_i))|
\end{equation}
forces $|H_i:U\cap H_i|=|UH_i:U|=1$, that is, $H_i\subseteq U$. In particular, $a_i\in U$, along with all its conjugates, by normality of $U$. But the generators of $G$ are all conjugates to some $a_i$, whence $G\subseteq U$, provided \eqref{eq:blaaaaah} holds for all $i=1,\dots,c$. The reverse implication is obvious.
\end{proof}
\begin{rem}
Proposition \ref{prop:analogue of cor3.8} is somehow weaker than Proposition \ref{prop:cor3.8neukirch}. In fact, the former deals with links in $S^3$ only. According to \cite{morishita}, $S^3$ is the manifold analogue of the field $\Q$. Hence, translating Proposition \ref{prop:analogue of cor3.8} to number theory amounts at fixing $F=\Q$.

On the other hand, the definition of link admits an obvious generalisation as embedding of disjoint copies of $S^1$ into an arbitrary $3$-manifold.
It would be interesting to understand to which $3$-manifolds Proposition \ref{prop:analogue of cor3.8} extends.
\end{rem}

\section[Homotopy of classifying spaces]{The homotopy of the quotients of the classifying spaces}
From now until the end of the chapter, knots will tacitly assumed to be prime.

Let $G$ be a prime knot group and $U\trianglelefteq G$ a normal subgroup of finite index. Our aim is to describe the space $E_\G G/U$.
The quotient space $EG/U$ has finitely many boundary path-connected components $T_1,\dots,T_r$, each homeomorphic to a torus. It remains to understand how the extra components $Cyl(\pi\colon\R^2\to\R)$ are transformed by quotienting out the $U$-action.
Let $H=\gen{a,l}\cong\Z^2$ be the peripheral subgroup of the chosen meridian. Then $U\cap H$ is a finite-index sublattice of $H$, whence it contains a nontrivial power of $a$; in fact, one can always choose a $\Z$-basis whose first vector is a power of $a$, as follows. Let $\{v,w\}$ be an arbitrary $\Z$-basis of $U\cap H$ and let $e=\min\{n\in\N\setminus\{0\}\mid a^n\in\spn_\Z\{v,w\}\}$. Then, expressing $a^e$ as a $\Z$-linear combination $\alpha v+\beta w$, necesarily $gcd(\alpha, \beta)=1$. As a consequence of Euclid's Algorithm (cf. \cite[Book 7, Proposition 1]{euclid}), there are $\gamma,\delta\in\Z$ such that 
\[
\mathrm{det}\begin{pmatrix} \alpha & \gamma \\ \beta & \delta \end{pmatrix}=\alpha\delta-\beta\gamma=1.
\]
The condition on the determinant says that $a^e$ and $u=v^\gamma w^\delta$ form a $\Z$-basis for $U\cap H$, that is, $U\cap H=\gen{a^e}\times\gen{u}$.
Now take the extra component $Cyl(\pi\colon EH\to EH^{\gen a})$ corresponding to $H$ (recall that the extra components, as well as the boundary components of $EG$, are indexed by the cosets of $G/H$). Quotienting first by $\gen{a^e}$ gives an infinite solid cylinder with axis $EH^{\gen a}$ and meridian $a^e$. Quotienting again by $\gen u$ provides a solid torus. The fact that in general $u$ is not orthogonal to $a$ is reflected in the configuration of the lines on the surface of this torus. In particular, if $u=a^pl^q$, the ``old'' longitude $l$ is wrapped around the axis by a constant slope of $-p/e$ (and one needs $q$ old longitudes to run around the whole torus).
This accounts for a solid torus attached to the boundary component of $EG/U$ originating from $H$ (a word of warning: the quotient map in general identifies several boundary components of $EG$ to a singular boundary component of $EG/U$). Clearly the same reasoning works for the other boundary components of $EG/U$.
Summing up, the space $E_\G G/U$ is obtained by gluing a closed solid torus $V_k\simeq Cyl(\pi\colon\R^2\to\R)/U$ to each boundary component of $EG/U$, in such a way that $T_k=\partial V_k$ $(k=1,\dots,r)$.

\subsection{The proof of Proposition A}
Thanks to the preceding discussion, the fundamental group of $E_\G G/U$ can be computed via iterated applications of Seifert and Van Kampen's Theorem. Suppose a presentation of $U$ is given, with set of generators $gen(U)$ and set of relators $rel(U)$.
Let $\pi_1(T_k)=\gen{m_k,l_k\mid m_kl_km_k^{-1}l_k^{-1}}\simeq \Z^2$ and $\pi_1(V_k)=\gen{\ell_k\mid\; }\simeq \Z$. Finally, let $i_k\colon\pi_1(T_k)\to\pi_1(BU)$ and $j_k\colon\pi_1(T_k)\to\pi_1(V_k)$ be the group homomorphisms induced by the inclusion maps. Then,
\[
\pi_1(E_\G G/U)= (\dots(\pi_1(BU)\ast_{\pi_1(T_1)}\pi_1(V_1))\ast\dots) \ast_{\pi_1(T_r)}\pi_1(V_r)
\]
(where one has obviously to consider small open neighbourhoods of $BU$ and each $V_k$ that deformation-retract onto $BU$ and $V_k$ respectively, as well as their intersection deformation-retracts onto $T_k$. This is always possible thanks to tameness).
Hence, $\pi_1(E_\G G/U)$ admits the presentation
\[
\begin{aligned}
&\gen{gen(U),\ell_1,\dots,\ell_r\mid rel(U),\; i_k(m_k)j_k^{-1}(m_k),\; i_k(l_k)j_k^{-1}(l_k),\; k=1,\dots,r}\\
=&\gen{gen(U),\ell_1,\dots,\ell_r\mid rel(U),\; i_k(m_k), \;\ell_kj_k^{-1}(l_k),\; k=1,\dots,r}\\
=&\gen{gen(U)\mid rel(U),\; i_k(m_k),\; k=1,\dots,r}
\end{aligned}
\]
(the last equality follows from Tietze's Theorem on equivalence of presentations).
Therefore, there is a short exact sequence of groups
\begin{equation}\label{eq:ses of groups}
\xymatrix{
1\ar[r]&\gen{\gen{i_1(m_1),\dots,i_r(m_r)}}\ar[r]&U\ar[r]&\pi_1(E_\G G/U)\ar[r]&1
}
\end{equation}
The elements $i_k(m_k)$ are representatives of the $U$-conjugacy classes of the powers of generators of $G$ that lie in $U$. We will denote 
\[
 M_U=\gen{\gen{i_1(m_1),\dots,i_r(m_r)}}.
\]
\subsection{Branched coverings of knot spaces}
Given a prime knot group $G$, every normal subgroup of finite index $U\id G$ induces a finite regular unbranched cover $EG/U$ and an associated branched cover $E_\G G/U$. By the construction of $E_\G G$ (see \eqref{eq:prime diagram explicit} and the beginning of this section),% the latter is obtained by filling some torus holes in the former. In particular, 
\begin{equation}\label{eq:branched included unbranched}
EG/U\subseteq E_\G G/U.
\end{equation}
Proposition A makes the link between unbranched and associated branched covers precise at the level of fundamental groups. Here we present a couple of related results.%Equivalently, quotienting the classifying spaces of knot groups for the family of the meridians by the action of normal subgroups of finite index. 
\begin{lem}\label{lem:proper disc action}
The group $U/M_U$ acts properly discontinuously on $E_\G G/M_U$ with quotient space $E_\G G/U$.
\end{lem}
\begin{proof}
The action is induced by the $G$-action on $E_\G G$, as follows:
\[
uM_U\bullet M_Ux=M_Uux
\]
for $u\in U, x\in E_\G G$.
The claim on the quotient space is obvious, so it remains to show proper discontinuousness.

Let $\overline x\in E_{\G}G/M_U$, that is, $\overline x=M_Ux$ for a $x\in E_{\G}G$. Let us set $\mathrm{Fix}_\G G=\sqcup_{C\in\G}(E_\G G)^C$ to simplify notation. If $\stab_G(x)=1$, i.e., $x\!\in\! E_{\G}G\setminus(\mathrm{Fix}_\G G)$, then by the construction of $E_\G G$ we can enlarge $EG\subseteq E_\G G$ to a $G$-CW-complex $Y\subseteq E_\G G$ which is $G$-homeomorphic to $EG$ and contains $x$ in its interior (this amounts to cutting away a small open neighbourhood of $\mathrm{Fix}_\G G$). Hence it is sufficient to consider only the two cases, $x\in EG$ and $\stab_G(x)\neq 1$.

Suppose that the $U/M_U$-action on $E_{\G}G/M_U$ does not satisfy proper discontinuousness, that is, there exist a $\overline x=M_Ux\in E_{\G}G/M_U$ (some $x\in E_\G G$) such that, denoting with $B$ a generic neighbourhood of $\overline x$,
\begin{equation}\label{eq:not prop disc}
\forall B : \exists uM_U\in U/M_U\setminus\{M_U\}: B\cap uM_UB\neq\emptyset.
\end{equation}
Letting $\tau\colon E_\G G\to E_\G G/M_U$ be the canonical quotient map, one has $\tau^{-1}(B)$ $=\bigcup_{m\in M_U}mD$ and $\tau^{-1}(uM_UB)=\bigcup_{m\in M_U}muD$, for a suitable neighbourhood $D$ of $x$.

In the first case ($x\in EG$), the neighbourhood $B$ can be chosen so small that $D$ be disjoint from all its translates $uD, u\in U\setminus\{1\}$ because the $G$-action on $EG$ is properly discontinuous by definition. This contradicts \eqref{eq:not prop disc}.
As regards the second case ($\stab_G(x)\neq 1$), by construction the stabilizer is infinite cyclic and generated by a meridian, say $\stab_G(x)=\gen{a}$. The obstruction in order to apply the same reasoning as before is that $D$ cannot be chosen to be disjoint from some of its translates. But, as long as $B$ is small enough, the only troubled translates of $D$ are the translates by powers of $a$, which are identified with $1$ when passing to the quotient by $M_U$.
\end{proof}

\begin{prop}[the Easter Beer Proposition]\label{prop:easter beer}
The canonical quotient map $\widehat\pi\colon E_{\G}G/M_U\to E_{\G}G/U$ is the universal covering map of $E_\G G/U$.
\end{prop}
\begin{proof}
%First of all, notice that we can characterize $M_U$ as the group generated by $\{^gm^k\mid g\in G,k\in \Z\}\cap U$.Hence, viewing $M_U$ as a function of $U$, $M_{M_U}=M_U$.
The group $U$ is finitely generated, being a finite-index subgroup of the knot group $G$. Hence $U/M_U$ is finitely generated too. Moreover, $U/M_U$ is residually finite, being the fundamental group of the compact $3$-manifold $E_\G G/U$ (cf. \cite{3mfdgp}). Thus, $U/M_U$ is Hopfian (cf. e.g. \cite[Section~4.1]{neumann}), i.e., it is not isomorphic to any of its proper subgroups.

But $U/M_U$ is isomorphic to the group $deck(\widehat\pi)$ of deck transformations of the covering $\widehat\pi\colon E_{\G}G/M_U\to E_{\G}G/U$ (Lemma \ref{lem:proper disc action} and \cite[Proposition 1.40]{hatcher}). On the other hand, $deck(\widehat\pi)\cong\pi_1(E_\G G/U)/\widehat\pi_\ast\left(\pi_1(E_\G G/M_U)\right)$, where $\widehat\pi_\ast$ is the homomorphism induced by $\widehat\pi$ on the fundamental groups (cf. \cite[Proposition 1.39]{hatcher}).

Together with Proposition A, this means
\[
{U}/{M_U}\cong \frac{{U}/{M_U}}{\widehat\pi_\ast\left(\pi_1(E_\G G/M_U)\right)}.
\]
Then, by Hopf property, $\widehat\pi_\ast\left(\pi_1(E_\G G/M_U)\right)=1$, whence the simple connectedness of $E_\G G/M_U$.
\end{proof}

\subsection{The proof of Theorem B}
Now we have all the ingredients needed to prove Theorem B, that we report here for the reader's convenience.
\begin{thmB}
Let $G$ be a prime knot group and let $U\id G$ be a normal subgroup of finite index. Then the following are equivalent.
\begin{enumerate}
\item $U=M_U$;
\item The canonical projection $\widehat\pi\colon E_\G G/M_U\to E_\G G/U$ is a trivial covering;
\item The canonical projection $\pi\colon E G/M_U\to E G/U$ is a trivial covering;
\item $\pi_1(E_\G G/U)=1$;
\item $E_\G G/U\cong S^3$;
\end{enumerate}
\end{thmB}
\begin{proof}
Clearly, (1)$\Rightarrow$(2). The implication (2)$\Rightarrow$(3) holds by restriction, see \eqref{eq:branched included unbranched}. Moreover, (3)$\Rightarrow$(1) by the Galois correspondence between coverings of $BG$ and subgroups of $G$ (cf. \cite[Theorem 1.38]{hatcher}. The implication (1)$\Leftrightarrow$(4) is given by the exact sequence \eqref{eq:ses of groups} (but it also follows from Proposition \ref{prop:easter beer}). Finally, (4)$\Leftrightarrow$(5) is Poincar\'e conjecture (now Hamilton-Perelman's Theorem, as Perelman himself would probably like to call it).
\begin{ex}[finite cyclic coverings and Fox completions]
Let $K\subseteq S^3$ be a knot, with tubular neighbourhood $V$, $X=S^3\setminus V$ and $G=G_K$ the knot group. Recall from Chapter \ref{chp:knotTheory} that $G/[G,G]$ is infinite cyclic generated by the class of a meridian $a$ and the winding number isomorphism $W\colon G/[G,G]\to\Z$ maps that class to $1$. The kernel $U_n$ of the composition
\[
\xymatrix@C+12pt{
G\ar^-{\mathrm{lk(\slot,K)}}[r] & \Z\ar^-{\mathrm{mod}\, n}[r] & \Z/n\Z
}
\]
is a normal subgroup of finite index in $G$.
Since 
\[
\frac{G/[G,G]}{U_n/[G,G]}\cong \frac{G}{U_n}\cong \Z/n\Z
\]
we have a description of $U_n$ as the extension
\[
1\to[G,G]\to U_n\to\gen{a^n}\to 1.
\]
with $a$ a meridian and $\gen{a^n}\cong\Z$.

The covering $\rho_n\colon X_n\to X$ associated to $U_n$, i.e., the unique covering of $X$ with group of deck transformations isomorphic to $\Z/n\Z$, is called the \emph{$n$-fold cyclic covering} of $X$ (or of $K$).
By construction, the covering space $X_n$ has only one boundary component and, letting $H=\gen{a,l}$ be the peripheral subgroup associated to the meridian $a\in G$, $U_n\cap H=\gen{a^n}\times\gen l$. Referring to the notation used at the beginning of this chapter, this means $e=n=|G:U_n|$, $f=r=1$.

Gluing one solid torus to $X_n$ along its meridian $a^n$, another one to $X$ along $a$, and extending in the obvious way the projection map $\rho_n$ to the manifolds so obtained, one gets the \emph{$n$-fold cyclic branched covering} of $K$
\[
\widehat{\rho_n}\colon \widehat{X_n}\to S^3.
\]
Some authors, for example Morishita (cf. \cite[Example 2.14]{morishita}), call this construction the \emph{Fox completion} of $\rho_n$.

Therefore, the domain of the Fox completion of the $n$-fold cyclic covering of $K$ is a model for the space $E_\G G/U_n$. Its fundamental group is 
\[
\pi_1(E_\G G/U_n)\cong U_n/\gen{a^n}.
\]

To be more concrete, let 
\[
G=\gen{a,b,c\mid ab=bc=ca}=\gen{a,b,c\mid aba=bab}
\]
be the trefoil knot group. Following \cite[Example 2.1]{silwil}, we can describe the commutator subgroup as
\[
[G,G]=\gen{a^j(a^{-1}b)a^{-j}\mid j\in\Z}
\]
and using Tietze moves, we can reduce the presentation to
\[
[G,G]=\gen{a^{-1}b,a^{-2}ba}.
\]
This is confirmed by a theorem of Stallings (cf. \cite[Section 5.A]{burzie}), saying that the commutator subgroup of a fibred knot group is free on $2g$ generators, where $g$ is the genus of the knot.

An interrogation of MAGMA Online Calculator
\begin{verbatim}http://magma.maths.usyd.edu.au/calc/\end{verbatim}gives
\vskip3em
\small
\begin{verbatim}
% QUESTION
G<a,b>:=Group<a,b|a*b*a=b*a*b>;
H:=sub<G|a^-1*b,a^-2*b*a, a^3>; % sgp generated by [G,G] and a^3
U_3:=H^G;                       % normal closure
U_3;

% ANSWER
					Finitely presented group U_3 on 3 generators
					Index in group G is 3
					Generators as words in group G
					U_3.1 = b * a^-1
					U_3.2 = b^-1 * a
					U_3.3 = a^3

% QUESTION
U3:=Rewrite(G,U_3); % presentation of U_3, just for personal culture
U3;

% ANSWER
					Finitely presented group U3 on 3 generators
					Generators as words in group G
					  U3.1 = b * a^-1
					  U3.2 = a * b * a
					  U3.3 = a^-1 * b
					Relations
					  U3.1^-1 * U3.2 * U3.1^-1 * U3.2^-1 = Id(U3)
					  U3.2 * U3.3^-1 * U3.2^-1 * U3.3^-1 = Id(U3)

% QUESTION
K:=sub<G|a^3>;
M_3:=K^G;
MU3:=Rewrite(G,M_3);
MU3;

% ANSWER
					Finitely presented group MU3 on 4 generators
					Generators as words in group G
					  MU3.1 = a^3
					  MU3.2 = b^3
					  MU3.3 = (a * b * a^-1)^3
					  MU3.4 = (b * a * b^-1)^3
					Relations
					  MU3.1^-1 * MU3.3^-1 * MU3.2^-1 * MU3.1 * MU3.4 * MU3.2 * 
					    MU3.3 * MU3.4^-1 = Id(MU3)
					  MU3.4 * MU3.2 * MU3.3 * MU3.1 * MU3.2^-1 * MU3.4^-1 * 
					    MU3.1^-1 * MU3.3^-1 = Id(MU3)
					  MU3.1 * MU3.4 * MU3.1^-1 * MU3.3^-1 * MU3.2^-1 * MU3.4^-1 * 
					    MU3.2 * MU3.3 = Id(MU3)

% QUESTION
Index(U3,MU3);

% ANSWER
					8

% QUESTION
Q<x,y,z>:=U3/MU3;
Q;

% ANSWER
					Finitely presented group Q on 3 generators
					Relations
					  x^-1 * y * x^-1 * y^-1 = Id(Q)
					  y * z^-1 * y^-1 * z^-1 = Id(Q)
					  x^-1 * y * z^-1 = Id(Q)
					  x * y * z = Id(Q)
					  y * z * x = Id(Q)
					  y * z^-1 * x^-1 = Id(Q)
\end{verbatim}
\normalsize
The map $x\mapsto i,\;y\mapsto j,\;z\mapsto \overline{k}$ provides an isomorphism between $Q$ and the group 
\[
Q_8=\gen{i,j,k\mid i^2=j^2=k^2=ijk}
\]
of the units of the quaternions. 
Summing up, we have obtained a regular branched covering of $S^3$, branched over the trefoil knot, with fundamental group $\pi_1(E_\G G/U_3)\cong Q_8$.
\end{ex}
\end{proof}

\begin{rem}\label{rem:structure of M_U}
It is quite difficult to give an explicit description of $M_U$. However, there are some restrictions on its structure. Since knot groups are coherent (as a consequence of Scott's Theorem, cf. \cite{scott}), $M_U$ is either infinitely generated or finitely presented. In the former case, the index $|U:M_U|$ clearly has to be infinite. In the latter case, one can apply the results in to obtain the following Tits-like alternative: either the index $|U:M_U|$ is finite or $M_U$ is free. In fact,  the groups of nontrivial knots have cohomological dimension $2$. Since $U$ has finite index by hypothesis, it has cohomological dimension $2$ as well (cf. \cite[Section VIII.2]{brown}). If $|U:M_U|$ is finite, then also $M_U$ has cohomological dimension $2$, hence it is not free. On the other side, if $|U:M_U|=\infty$, then by \cite{bieri1978} $M_U$ has cohomological dimension $1$, hence it is free.
\end{rem}

\chapter[Homology and Poitou-Tate sequence]{(Co)homology of the classifying spaces and Poitou-Tate exact sequence}
Throughout the chapter, let $K\subseteq S^3$ be a fixed knot, with knot group $G$ and a tubular neighbourhood $V$, and set $X=S^3\setminus V$, $T=\partial V$. We denote $\widetilde X$ the universal cover of $X$. Let $a\in G$ be our favourite meridian, $l$ the associated longitude and $H=\gen{a,l}$ the corresponding peripheral subgroup.
\section{Poincar\'e-Lefschetz duality}
The aim of this section is to prove that there is a short exact sequence
\begin{equation}\label{eq:geometric sec}
\xymatrix{
0\ar[r] & H^2(G,\Z[G])\ar^-{\beta}[r] & \ind_H^G(\Z)\ar^-{\alpha}[r] & \Z\ar[r] & 0
}
\end{equation}
In fact, the CW-pair $(\widetilde{X},\partial\widetilde{X})$ gives rise to a long exact sequence (cf. \cite[\s~2.1]{hatcher})
$$
\xymatrix@C-2pt{
\dots\ar[r] & H_2(\widetilde{X},\partial\widetilde{X})\ar[r]^-{\xi_2} & H_1(\partial\widetilde{X})\ar[r]^-{\iota_1} & H_1(\widetilde{X})\ar[r]^-{\pi_1} & %\\
 H_1(\widetilde{X},\partial\widetilde{X})\ar[r] & \\
\ar[r]^{\xi_1} & H_0(\partial\widetilde{X})\ar[r]^-{\iota_0} & H_0(\widetilde{X})\ar[r]^-{\pi_0} & H_0(\widetilde{X},\partial\widetilde{X})\ar[r] & 0
}
$$

By simply-connectedness of universal covers $H_0(\widetilde{X})\cong\Z$ and $H_1(\widetilde{X})=0$.

Moreover, the terms $H_\bullet(\widetilde{X},\partial\widetilde{X})$ are the homology of the relative chain complex
$$
\xymatrix{
\dots\ar[r] & \dfrac{C_2(\widetilde{X})}{C_2(\partial\widetilde{X})}\ar[r]^-{\chi_2} & \dfrac{C_1(\widetilde{X})}{C_1(\partial\widetilde{X})}\ar[r]^-{\chi_1} & \dfrac{C_0(\widetilde{X})}{C_0(\partial\widetilde{X})} \ar[r]^-{\chi_0}& 0
}
$$
If we consider a generator $y$ of $C_0(\widetilde X)$, there is a path $s$ in $\widetilde{X}$ such that $s(0)=y$ and $s(1)$ is the generator of $C_0(\partial \widetilde X)$. Then $\chi_1([s])=[y]$. This shows that $\chi_1$ is surjective or, in other words, $H_0(\widetilde{X},\partial\widetilde{X})=0$.

Summing up, we obtain the short exact sequence
$$
\xymatrix{
0\ar[r] & H_1(\widetilde{X},\partial\widetilde{X})\ar[r]^-{\xi_1} & H_0(\partial\widetilde{X})\ar[r]^-{\iota_0} & \Z\ar[r] & 0.
}
$$

The connected components of $\partial\widetilde X$ are planes in $1$-to-$1$ correspondence with the set $G:H$. These planes are permuted by the action of $G$ in such a way that $H$ is the isotropy group of the index corresponding to one of them, hence $H_0(\partial\widetilde X)\cong\ind^G_H(\Z)$ (cf. \cite[\s~III.5]{brown}).

Lefschetz-Poincar\'e duality (cf. \cite[\s~3.3]{hatcher}) yields $H_1(\widetilde{X},\partial\widetilde{X})\cong H^2_c(\widetilde{X})$ (cohomology with compact support). But one has also $H^2_c(\widetilde{X})\cong H^2(X,\Z[G])$ (cf. \cite[\s~1]{eckmann76}). Finally, by asphericity of knot complements (recall Papakyriakopoulos's Sphere Theorem, \cite{papa}),
\[H^2(X,\Z[G])\cong H^2(G,\Z[G])\]
and the short exact sequence (\ref{eq:geometric sec}) is established, with $\alpha=\iota_0$ and $\beta$ given by the composition of $\xi_1$ with the isomorphism $H^2(G,\Z[G])\to H_1(\widetilde{X},\partial\widetilde{X})$.

\section{The proof of Theorem C}
Our aim is to find a suitable algebraic counterpart to the topologically flavoured sequence (\ref{eq:geometric sec}).
We obtain a projective (in fact, free) left $H$-module resolution of the trivial $H$-module $\Z$ from a cellular chain complex of the universal cover $\R^2$ of $T$:
\begin{equation}\label{eq:H resol Z 2}
 \xymatrix{
0\ar[r] & \Z[H]\langle\ul r\rangle\ar[r]^-{\partial_2} & \Z[H]\langle\ul m\rangle\oplus\Z[H]\langle\ul l\rangle\ar[r]^-{\partial_1} & \Z[H]\langle\ul0\rangle\ar@{->>}[d]\\
&&&\Z
}
\end{equation}
where $\pi_1(T,x)=\gen{m,l\mid r}$, $r=mlm^{-1}l^{-1}$ and
\begin{eqnarray*}
\partial_2(\ul r)&=&\ul{m}+m\ul{l}-mlm^{-1}\ul m-mlm^{-1}l^{-1}\ul l=(1-l)\ul{m}+(m-1)\ul l,\\
\partial_1(\ul m)&=&(m-1)\ul 0,\\
\partial_1(\ul l)&=&(l-1)\ul 0.
\end{eqnarray*}
Here some explanations about the notational conventions may be necessary. We underlined the generators of our modules in order to distinguish them from the corresponding elements of the group algebra $\Z[H]$; moreover, we introduced the formal symbol $\ul0$ to denote the $H$-generator of the $0$-chains. Similar conventions will be used for the other resolutions introduced in this chapter.

The functor $\ind^G_P(\slot)=\Z[G]\otimes_{\Z[H]}\slot$ is exact and maps projective $H$-modules to projective $G$-modules, thus we get a projective $G$-module resolution $(Q_\bullet,\delta_\bullet)$ of $\ind^G_H(\Z)$:
\begin{equation}\label{eq:G resol ind(Z)}
 \xymatrix{
0\ar[r] & \Z[G]\langle\ul r\rangle\ar[r]^-{\delta_2} & \Z[G]\langle\ul m\rangle\oplus\Z[G]\langle\ul l\rangle\ar[r]^-{\delta_1} & \Z[G]\langle\ul 0\rangle\ar@{->>}[d]\\
&&&\ind^G_H(\Z).
}
\end{equation}

The same technique yielding \eqref{eq:H resol Z 2} works for finding a projective $G$-module resolution $(P_\bullet,\gamma_\bullet)$ of $\Z$ (recall from Chapter \ref{chp:knotTheory} that nontrivial knot groups have cohomological dimension $2$):
\begin{equation}\label{eq:G resol Z}
 \xymatrix{
0\ar[r] & \bigoplus_{i=1}^{d-1}\Z[G]\langle\ul {r_i}\rangle\ar[r]^-{\gamma_2} & \bigoplus_{i=1}^d\Z[G]\langle\ul {a_i}\rangle\ar[r]^-{\gamma_1} & \Z[G]\gen{\ul1}\ar@{->>}[d]\\
&&&\Z,
}
\end{equation}
where $G=\langle a_1,\dots,a_n\mid r_1,\dots,r_{n-1}\rangle$ is a Wirtinger presentation of the knot group.
The boundary maps are as follows: any relator $r_i$ has the shape $a_ja_ka_l^{-1}a_k^{-1}$ and 
\begin{eqnarray*}
\gamma_2(\ul{a_ja_ka_l^{-1}a_k^{-1}})&=&\ul{a_j}+a_j\ul{a_k} - a_ja_ka_l^{-1}\ul{a_l}-a_ja_ka_l^{-1}a_k^{-1}\ul{a_k}\\
\gamma_1(\ul{a_i})&=&(a_i-1)\ul 0.
\end{eqnarray*}

The two maps $\delta_2$ and $\gamma_2$ are defined in order to codify the loops in the Cayley graphs of $H\cong\Z^2$ (with generating set $S=\{m,l,m^{-1},l^{-1}\}$) and $G$ (with generating set $S=\{a_1,a_i^{-1}\mid i=1,\dots,d\}$) respectively.
It is worth noting that they act on the generators of the respective domains like some sort of module-theoretic version of Fox derivations (cf. \cite[\S~1]{foxfree}), which inspires the following
\begin{defn}\label{def:Fox module derivative}
Let $R$ be a ring and $G$ be a group. Let $X=\{g_i\mid i\in I\}$ be a set of generators of $G$ and $\mathcal{F}(X)$ the free group on $X$. The \emph{Fox module-derivative} is the map $\Delta\colon \mathcal{F}(X)\to \oplus_{i\in I}R[G]\gen{\ul{g_i}}$ defined recursively by
\begin{equation}\label{eq:free calculus axioms}
\left\{
\begin{array}{lcl}
\Delta(g_i)=\ul{g_i},\\
\Delta(g_i^{-1})=-g_i^{-1}\cdot\Delta(g_i),\\
\Delta(vw)=\Delta(v) + v\cdot\Delta(w).
\end{array}
\right.
\end{equation}
\end{defn}

\begin{rem}\label{rem:involution}
$\Z[G]$ comes equipped with the involution (antipode)
\[
S\colon\Z[G]\to \Z[G],\; S(g)=g^{-1}
\]
that transforms left $G$-modules into right ones and viceversa. In more detail, given a left $G$-module $L$ with action $\cdot$, we denote $L\dx$ the right $G$-module whose underlying abelian group is $L$ endowed with the action $\star:L\times\Z[G]\to L$, $l\star g=g^{-1}\cdot l$. Viceversa,  given a right $G$-module $R$ with action $\star$, we denote $\sx R$ the left $G$-module whose underlying abelian group is $R$ endowed with the action $\cdot:\Z[G]\times R\to R$, $g\cdot r=r\star g^{-1}$.
\end{rem}

In particular, for any left $G$-module $M$ one can define a \emph{twisted} version of the dual module. Its underlying abelian group is
\[
M\ox=\hom^S_G(M,\Z[G])=\{f\colon M\to\Z[G]\mid f(am)=f(m)S(a)\},
\]
which becomes a left $G$-module via the action $a\cdot f\colon m\mapsto af(m)$.
Note that the natural $G$-action on the ``classical'' dual module, that is $M^*=\hom_G(M, \Z[G])=\{f\colon M\to\Z[G]\mid f(am)=af(m)\}$, is a \emph{right} action given by $f\star g\colon m\mapsto f(m)g$. We can twist it by the antipode to obtain a \emph{left} $G$-module
\[
\sx M^*=\hom_G(M, \Z[G])
\]
with action $g\cdot f= f\star g^{-1}\colon m\mapsto f(m)g^{-1}$.
The two versions of left dual module are naturally isomorphic through the composition with the antipode:
\[
\xymatrix@C+6pt{
\hom^S_G(M,\Z[G])\ar[r]^-{S\circ\slot} &  \sx\hom_G(M, \Z[G])
}
\]

Applying the functor $\slot\ox=\hom^S_G(\slot,\Z[G])$ to the augmented chain complex $P_\bullet\stackrel{\epsilon}{\to}\Z\to 0$ we obtain the chain complex of left $G$-modules
\begin{equation}\label{eq:dualchaincomplex}
\xymatrix{
 0\ar[r] & \sx\Z\ox\ar^-{\epsilon\ox}[r] & \Z[G]\langle\ul{1^*}\rangle\ar^-{\widetilde{\gamma}\ox_2}[r] & \bigoplus_{i=1}^d\ \Z[G]\langle \ul{a_i^*}\rangle\ar^-{\widetilde{\gamma}\ox_1}[r] & \bigoplus_{j=1}^{d-1}\ \Z[G]\langle \ul{r_j^*}\rangle.
}
\end{equation}
The term $\Z\ox=\hom^S_G(\Z,\Z[G])$ is $0$. In fact, as a group homomorphism $f\colon\Z\to\Z[G]$ sends every $z\in \Z$ to a sum $\sum z_{g_i}g_i$ with $z_{g_i}\neq 0$ for finitely many indices $i$, $f$ commutes with the $G$-actions if and only if for all $z\in\Z$ and $g\in G$ one has
\[
\sum z_{g_i}g_i=f(z)=f(gz)=f(z)g^{-1}=\sum z_{g_i}(g_ig^{-1}).
\]
On the other hand, $G$ is infinite and the multiplication action of $G$ on itself is free, which means that if some $z_{g_i}$ is nonzero, then all the infinitely many $\{z_{g_ig} \mid g\in G\}$ are.
This forces all $z_{g_i}$ to be $0$.
The homology of the chain complex is by construction the cohomology of $G$ with coefficients in the group ring $\Z[G]$. According to \cite{bieck}, $H^0(G,\Z[G])=H^1(G,\Z[G])=0$, so the complex \eqref{eq:dualchaincomplex} is exact everywhere but at the end, where its homology is $H^2(G,\Z[G])$. But since the dual functor $\slot^*=\hom_G(\slot,\Z[G])$ maps finitely generated projective left $G$-modules to finitely generated projective right $G$-modules (cf. for instance \cite[Prop.~I.8.3]{brown}), the functor $\slot\ox$ maps finitely generated projective left $G$-modules to finitely generated projective left $G$-modules. Thus, we actually get a projective resolution $(\widetilde{P}^\circledast_\bullet, \widetilde{\gamma}^\circledast_\bullet)$ of the left $G$-module $H^2(G,\Z[G])\ox$.
\begin{equation}\label{eq:P circledast}
 \xymatrix{
 0\ar[r] & \Z[G]\langle \ul{1^*}\rangle\ar^-{\widetilde{\gamma}\ox_2}[r] & \bigoplus_{i=1}^d\ \Z[G]\langle \ul{a_i^*}\rangle\ar^-{\widetilde{\gamma}\ox_1}[r] & \bigoplus_{j=1}^{d-1}\ \Z[G]\langle \ul{r_i^*}\rangle\ar@{->>}[d]\\
 &&&H^2(G,\Z[G])\ox.
}
\end{equation}
Some basic linear algebra provides the explicit behaviour of the boundary maps, as follows. We denote $r_{jkh}$ the Wirtinger relation $a_ja_ka_h^{-1}a_k^{-1}$ for some indices $j,k,h\in\{1,\dots,d\}$ (this notation is unambiguous once the generators are fixed) and we define
$$R=\left\{(j,k,h)\in\{1,\dots,d\}^3\mid r_{jkh}\mbox{ is a Wirtinger relation of }G\right\}.$$
Then
\begin{eqnarray}
\widetilde{\gamma}\ox_2(\ul{1^*})&=&\sum_{i=1}^d(a_i^{-1}-1)\ul{a_i^*}\\
\widetilde{\gamma}\ox_1(\ul{a_i^*})&=&\sum_{(k,h)\mid (i,k,h)\in R}\hspace*{-8pt}\ul{r^*_{ikh}}+\sum_{(j,h)\mid (j,i,h)\in R}\hspace*{-8pt}(a_j^{-1}-1)\ul{r^*_{jih}}-\sum_{(j,k)\mid (j,k,i)\in R}\hspace*{-8pt}a_k^{-1}\ul{r^*_{jki}}\nonumber
\end{eqnarray}
In principle, this resolution is concentrated in degrees $-2,-1,0$. One can then shift it by $+2$ to obtain a chain complex $(P^\circledast_\bullet, \gamma^\circledast_\bullet)=(\widetilde{P}^\circledast[2]_\bullet, \widetilde{\gamma}^\circledast[2]_\bullet)$ concentrated in non-negative degrees (cf. Section \ref{sec:derived}).

Now we have the following diagram with exact rows and exact rightmost column\par
\begin{equation}\label{eq:diagram of resolutions}
\xymatrix@R-8pt@C-8pt{\displaystyle
 &  &  &  & 0 \ar[d] \\
0 \ar[r]& \Z G\gen{\ul{1^*}}\ar[r] & \bigoplus_{i=1}^n\Z G\gen{\ul{a_i^*}}\ar[r]& \bigoplus_{j=1}^{n-1}\Z G\gen{\ul{r_j^*}}\ar@{>>}[r] & H^2(G,\Z G)\ar^-{\beta}[d] \\
0\ar[r] & \Z G\gen{\ul r}\ar[r] & \Z G\gen{\ul m}\oplus\Z G\gen{\ul l}\ar[r] & \Z G\gen{\ul 0}\ar@{>>}[r] & \ind^G_H\Z\ar^-{\alpha}[d] \\
0\ar[r] & \bigoplus_{j=1}^{n-1}\Z G\gen{\ul{r_j}}\ar[r] & \bigoplus_{i=1}^n\Z G\gen{\ul{a_i}}\ar[r] & \Z G\gen{\ul 1}\ar@{>>}[r] & \Z\ar[d]\\
 &  &  &  & 0.
}
\end{equation}
on which we can apply the functors $\mathrm{Tor}^G_\bullet(\underline{\phantom{M}},\Z)$. Taking into account that the resolution on the first row of Diagram \eqref{eq:diagram of resolutions} is the dual of the projective resolution of the trivial $G$-module $Z$ (third row), the resulting $9$-terms long exact sequence can be written as
\begin{equation}\label{eq:poitou-tate G}
\xymatrix{
0\ar[r] & H^0(G,\Z)\ar[r] & H_2(H,\Z)\ar[r] & H_2(G,\Z)\ar[dll] & \\
 & H^1(G,\Z)\ar[r] & H_1(H,\Z)\ar[r] & H_1(G,\Z)\ar[dll] & \\
 & H^2(G,\Z)\ar[r] & H_0(H,\Z)\ar[r] & H_0(G,\Z)\ar[r] & 0.
}
\end{equation}
This is strongly reminiscent of Poitou-Tate exact sequence for algebraic number fields.
It could seem that, apart from its great philosophical relevance, \eqref{eq:poitou-tate G} does not give any new information, since it reduces to
\[
\xymatrix{
0\ar[r] & \Z\ar[r] & \Z\ar[r] & 0\ar[dll] & \\
 & \Z\ar[r] & \Z\oplus\Z\ar[r] & \Z\ar[dll] & \\
 & 0\ar[r] & \Z\ar[r] & \Z\ar[r] & 0.
}
\]
But it is in conjunction with the unitarity of the restriction functor (see Proposition \ref{prop:res unitary}) that \eqref{eq:poitou-tate G} reveals its full power, as we shall now explain.

We can also dualise the resolution $Q_\bullet$ of $\ind^G_H(\Z)$. The same reasoning as for $P_\bullet$, together with the adjunction 
\[
 \hom_G(\ind^G_H(\Z),\Z[G])\cong\hom_H(\Z,\res^G_H(\Z[G])),
\]
leads to a fourth projective resolution $(Q_\bullet\ox, \delta_\bullet\ox)$ of the left $G$-module:
\begin{equation}\label{eq:Q circledast}
 \xymatrix{
 0\ar[r] & \Z[G]\langle \ul{0^*}\rangle\ar^-{\delta\ox_2}[r] & \Z[G]\langle \ul{m^*}\rangle\oplus\ \Z[G]\langle \ul{l^*}\rangle\ar^-{\delta\ox_1}[r] & \Z[G]\langle \ul{r^*}\rangle\ar@{->>}[d]\\
 &&&H^2(H,\res^G_H\Z[G]).
}
\end{equation}
The boundary maps are
\begin{eqnarray}
\delta\ox_2({\ul{0^*}})&=&(m^{-1}-1){\ul{m^*}}+(l^{-1}-1)\ul{l^*}\\
\delta\ox_1(\ul{m^*})&=&(1-l^{-1})\ul{r^*}\nonumber\\
\delta\ox_1(\ul{l^*})&=&(m^{-1}-1)\ul{r^*}\nonumber
\end{eqnarray}
One can define a chain map $\zeta_\bullet\colon Q_\bullet\to Q\ox_\bullet$ by
\begin{eqnarray}
\zeta_0(\ul{0})&=&\ul{r^*}\\
\zeta_1(\ul{m})&=&-m\ul{l^*}\nonumber\\
\zeta_1(\ul{l})&=&l\ul{m}\nonumber\\
\zeta_2(\ul{r})&=&-ml\ul{0^*}.\nonumber
\end{eqnarray}
This turns out to to be a self-duality in the category with duality of complexes of left $G$-modules $(\kom(G-\fr{M}od), \ox, \varpi)$ (see Definition \ref{def:duality and self-duality}), where the natural isomorphism $\varpi$ is given by
\[
 \varpi_M\colon M\to M\ox\! \ox,\qquad \varpi_M(x):f\to S(f(x))
\]
(so that, in our situation,
\[
 \varpi_Q(\ul 0)=0^{**},\;
 \varpi_Q(\ul m)=m^{**},\;
 \varpi_Q(\ul l)=l^{**},\;
 \varpi_Q(\ul r)=r^{**}).
\]

Indeed, $\zeta$ has inverse map
\begin{eqnarray}
\zeta^{-1}_0(\ul{r^*})&=&\ul 0\\
\zeta^{-1}_1(\ul{m^*})&=&l^{-1}\ul{l}\nonumber\\
\zeta^{-1}_1(\ul{l^*})&=&-m^{-1}\ul{m}\nonumber\\
\zeta^{-1}_2(\ul{1^*})&=&-l^{-1}m^{-1}\ul{r}\nonumber
\end{eqnarray}
and adjoint map $\zeta^\sharp_\bullet\colon Q_\bullet\ox\! \ox\to Q\ox_\bullet$ given by
\begin{eqnarray}
\zeta^\sharp_0(\ul{0^{**}})&=&-m^{-1}l^{-1}\ul{r^*}\\
\zeta^\sharp_1(\ul{m^{**}})&=&l^{-1}\ul{l^*}\nonumber\\
\zeta^\sharp_1(\ul{l^{**}})&=&-m^{-1}\ul{m^*}\nonumber\\
\zeta^\sharp_2(\ul{r^{**}})&=&\ul{0^*}.\nonumber
\end{eqnarray}

The map $\zeta$ fits into the diagram of bounded complexes of finitely generated projective $G$-modules
\begin{equation}\label{eq:self-dual triangles}
\xymatrix{
P\ox\ar@{.>}^-{?_1}[rr]\ar_-{Id_\bullet}[d] && Q\ar^-{\alpha_\bullet}[rr]\ar_-{\zeta_\bullet}[d] && P\ar@{.>}^-{?_2}[rr]\ar_-{(Id^*_\varpi)_\bullet}^-{=\,(\varpi_P)_\bullet}[d] && P\ox[1]\ar_-{Id[1]_\bullet}[d]\\
P\ox\ar^-{\alpha^*_\bullet}[rr] && Q\ox\ar@{.>}^{?_3}[rr] && P\ox\!\ox\ar@{.>}^{?_4}[rr] && P\ox[1]
}
\end{equation}
where $\alpha_\bullet$ is the map induced by $\alpha\colon\ind^G_H(\Z)\to\Z$ on the projective resolutions via the Comparison Theorem in Homological Algebra (cf. \cite[Section III.6]{maclane}.

Our goal is to complete this diagram to a commutative diagram of distinguished triangles (in the derived category 
of bounded complexes of finitely generated projective $G$-modules) that encodes a self-duality (cf. Definition \ref{def:triang duality and self-duality}) between the two rows.
Since $\zeta_\bullet$ is invertible, we can choose as $?_1$ the map $\zeta^{-1}_\bullet\circ\alpha^*_\bullet$, so that the leftmost square commutes.
Then the map $?_3$ ought to be $(\zeta_\bullet^{-1}\circ\alpha_\bullet^*)^*$, and in fact this choice would make the central square commute. Indeed, $(\zeta_\bullet^{-1}\circ\alpha_\bullet^*)^*=\alpha_\bullet^{**}\circ(\zeta_\bullet^{-1})^*$ and $(\zeta_\bullet^{-1})^*=(\zeta_\bullet^*)^{-1}$.
Moreover, the fact that $\zeta_\bullet$ is self-adjoint says that in the diagram
\[
 \xymatrix{
 & Q\ar^-{\alpha_\bullet}[rr]\ar_-{\zeta_\bullet}[ld]\ar^-{\varpi_Q}[dd] && P\ar_-{\varpi_P}[dd] \\
 Q\ar_-{(\zeta_\bullet^{-1})^*=(\zeta_\bullet^*)^{-1}}[rd]\ox && \\
 & Q\ox\!\ox\ar^-{\alpha^{**}_\bullet}[rr] && P\ox\!\ox
 }
\]
the left triangle commutes.
Finally, since $\varpi$ is a natural isomorphism $Id\to \slot^{**}$, we conclude that $\varpi_P\circ\alpha_\bullet=\alpha^{**}\circ(\zeta^{-1})^*\circ\zeta$.
As regards the rightmost square in Diagram \eqref{eq:self-dual triangles}, the map $?_2$ is uniquely determineed up to homotopy by imposing the first row to be a distinguished triangle. More precisely, $?_2$ should be a map $\xi_\bullet$ such that its homotopy class $[\xi]$ is the unique element of $\Ext_{\Z[G]}^1\left(H_0(P_\bullet), H_0(P_\bullet\ox)\right)$ (this is because the complexes $P_\bullet$ and $P_\bullet\ox$ are acyclic). The same holds for the map $?_4=:\vartheta_\bullet$: there is exactly one choice (up to homotopy) that makes the second row a distinguished triangle.
But $\varpi_P$ is invertible, and the map $Id[1]_\bullet\circ\xi_\bullet\circ\varpi_P^{-1}\colon P_\bullet\ox\!\ox\to P\ox[1]_\bullet$, which obviously makes the rightmost square commute, also turns the second row into an exact triangle, by the definition of $\varpi$, so $\vartheta_\bullet=Id[1]_\bullet\circ\xi_\bullet\circ\varpi_P^{-1}$ up to homotopy.
Summing up, we have proved 
\begin{lem}
$\dots\to P\ox_\bullet\to Q_\bullet\to P_\bullet\to P\ox[1]_\bullet\to\dots$ is a self-dual distinguished triangle in the derived category 
of bounded complexes of finitely generated projective left $G$-modules.
\end{lem}

If $U$ is any finite-index subgroup of $G$, then $\res^G_U$ is a unitary functor (cf. Proposition \ref{prop:res unitary}), so it maps the self-dual distinguished triangle
\[\dots\to P\ox_\bullet\to Q_\bullet\to P_\bullet\to P\ox[1]_\bullet\to\dots\]
in the derived category of bounded complexes of finitely generated projective left $G$-modules
to a self-dual distinguished triangle %s 
in the derived category of bounded complexes of finitely generated projective left $U$-modules. Therefore, we automatically get a more general version of the $9$-terms exact sequence \eqref{eq:poitou-tate G}, namely
\begin{equation}\label{eq:poitou-tate U}
\xymatrix{
0\ar[r] & H^0(U,\Z)\ar[r] & \coprod_{G/U\! H} H_2(H\cap U,\Z)\ar[r] & H_2(U,\Z)\ar[lld] & \\
 & H^1(U,\Z)\ar[r] & \coprod_{G/U\! H} H_1(H\cap U,\Z)\ar[r] & H_1(U,\Z)\ar[lld] & \\
 & H^2(U,\Z)\ar[r] & \coprod_{G/U\! H} H_0(H\cap U,\Z)\ar[r] & H_0(U,\Z)\ar[r] & 0.
}
\end{equation}

\section{Some steps towards Conjecture D}
Now select the subsequence in degree $1$ of the exact sequence \eqref{eq:poitou-tate U}: Conjecture D claims that the first cohomology and homology groups of $E_\G G/U$ fit perfectly at the ends of the subsequence.
\begin{conjD}
Let $G$ be a knot group, $H\leq G$ a peripheral subgroup and $U\id G$ a normal subgroup of finite index. Then there is an exact sequence of groups
\[
\xymatrix@C-6pt{
0\ar[r] & H^1(E_\G G/U,\Z)\ar[d]\\
 & H^1(U,\Z)\ar[r] & \coprod_{G/U\! H} H_1(H\cap U,\Z)\ar[r] & H_1(U,\Z)\ar[d] & \\
 &&&H_1(E_\G G/U,\Z)\ar[r] & 0.
}
\]
\end{conjD}

From now on, coefficients in homology and cohomology are understood to be $\Z$ unless otherwise specified.

Conjecture D is certainly true in the trivial cases, that is, if $U$ is a knot group and $E_\G G/U$ is the $3$-sphere, or equivalently, if $UH=G$ and $\pi_1(E_\G G/U)=1$. Indeed, in this case Diagram \eqref{eq:poitou-tate U} is formally identical to Diagram \eqref{eq:poitou-tate G} and $H^1(E_\G G/U)=H_1(E_\G G/U)=0$.

Next, consider the case in which $U$ is still a knot group, but in a generic $3$-manifold, that is, $UH=G$, but $\pi_1(E_\G G/U)$ is arbitrary.
Then, applying the abelianisation functor to the sequence \eqref{eq:ses of groups} one gets the exact sequence
\[
\xymatrix{
\gen m\ar^-{\widehat\alpha}[r] & H_1(U)\ar[r] & H_1(E_\G G/U)\ar[r] & 0.
}
\]
One can enlarge the first entry and get a still exact sequence
\begin{equation*}
\xymatrix{
H_1(H)\ar^-{\widetilde\alpha}[r] & H_1(U)\ar[r] & H_1(E_\G G/U)\ar[r] & 0
}
\end{equation*}
by defining
\[
\widetilde\alpha(m)=\widehat\alpha(m),\quad\widetilde\alpha(l)=0.
\]
This sequence, its $\mathrm{hom}$-dual
\[
\xymatrix{
0\ar[r] & H^1(E_\G G/U)\ar[r] &  H^1(U)\ar^-{\widetilde\alpha^*}[r] & H^1(H)\ar[r] & 0
}
\]
(where we used the Universal Coefficient Theorem for cohomology, cf. \cite[Section 3.1]{hatcher}, along with torsion-freeness of $0$th homology groups) and the degree $1$ subsequence of \eqref{eq:poitou-tate U} fit into a diagram with exact rows
\small
\begin{equation}
\xymatrix{
&&& H_1\hskip-0.7pt(\hskip-0.7pt H\hskip-0.7pt)\ar^-{\widetilde\alpha}[r]\ar@{=}[d] & H_1\hskip-0.7pt(\hskip-0.7pt U\hskip-0.7pt)\ar@{=}[d]\ar[r] & H_1\!\hskip-1pt\left(\hskip-1.5pt \frac{E_\G G}{U}\hskip-1.2pt\right)\ar[r] & 0\\
&& H^1\hskip-0.7pt(\hskip-0.7pt U\hskip-0.7pt)\ar@{=}[d]\ar^-{H_{\hskip-0.4pt 1}\hskip-0.7pt(\hskip-0.7pt\beta_\bullet\hskip-0.7pt)}[r] & H_1\hskip-0.7pt(\hskip-0.7pt H\hskip-0.7pt)\ar_-{\wr}^-{H_{\hskip-0.4pt 1}\hskip-0.7pt(\hskip-0.7pt\res^G_U\zeta_\bullet\hskip-0.7pt)}[d]\ar^-{H_{\hskip-0.4pt 1}\hskip-0.7pt(\hskip-0.7pt\alpha_\bullet\hskip-0.7pt)}[r] & H_1\hskip-0.7pt(\hskip-0.7pt U\hskip-0.7pt) &&\\
0\ar[r] & H^1\!\hskip-1pt\left(\hskip-1.5pt\frac{E_\G G}{U}\hskip-1.2pt\right)\ar[r] & H^1\hskip-0.7pt(\hskip-0.7pt U\hskip-0.7pt)\ar^-{\widetilde\alpha^*}[r] & H^1\hskip-0.7pt(\hskip-0.7pt H\hskip-0.7pt) &&&
}
\end{equation}
\normalsize
Suppose further that the longitude $l\in H$ is null-homologous in $BU$.
Then the upper square obviously commutes by construction. The lower one commutes because $H_1(\res^G_U\beta_\bullet)$ is exactly the dual map of $H_1(\res^G_U\alpha_\bullet)$ composed with the Poincar\'e duality isomorphism $H_1(\res^G_U\zeta_\bullet))^{-1}$, as expressed in Diagram \eqref{eq:self-dual triangles} and Proposition \ref{prop:res unitary}.
Summing up, Conjecture D is true for normal subgroups $U$ such that $BU$ has only one (torus) boundary component, provided the longitude of $\partial BU$ is null-homologous in $BU$.

\appendix
\titleformat
{\chapter} % command
[display] % shape
{\mdseries\huge\scshape} % format
{\centering Appendix \thechapter} % label
{0ex} % sep
{
    \rule{\textwidth}{1pt}
    \vskip1.2ex
    \centering
} % before-code
\chapter[Duality]{Duality in derived categories}
Here we briefly discuss some technical tools from category theory used in the previous chapter. They involve the construction of various categories whose objects are chain complexes over a fixed abelian category.
\section{Derived categories}\label{sec:derived}
Throughout this section, let $\fr{A}$ be an abelian category. We will construct the derived category of $\fr A$ in $3$ steps, each of which consists of creating a new, increasingly sophisticated category based on $\fr A$.
\begin{defn}
The \emph{category of complexes over} $\hskip-0.5pt\fr A\hskip-0.5pt$ is the category $\hskip-0.5pt\kom\hskip-1pt(\hskip-1pt\fr A\hskip-1pt)$ defined as follows: 
\begin{description}
\item[Objects] chain complexes with entries in $\fr A$, i.e.
\[
A=(A^\bullet, d_A^\bullet)=\dots\to A^{k-1}\stackrel{d_A^{k-1}}{\to} A^k\stackrel{d_A^{k}}{\to} A^{k+1}\to\dots
\]
with $A^k$ and $d_A^k$ in $\fr A$ and $d_A^{k}\circ d_A^{k-1}=0$ for all $k$ (the $d_A^k$s are called \emph{boundary maps}).
\item[Morphisms] chain maps, i.e.
\[
\phi^\bullet\colon A\to B = (\phi^k\colon A^k\to B^k)_k
\]
with $\phi^{k+1}\circ d_A^k=d_B^k\circ\phi^k$ for all $k$. We will simply write $\phi\colon A\to B$ when there is no danger of confusion.
\end{description}
\end{defn}
In particular (at least when $\fr A$ is concrete, which is always true in our setting) one can take the (co)homology of a complex $H^k(A)=\ker d_A^k/\mathrm{im} d_A^{k-1}$ and view it as a chain complex with $0$ boundary maps. A chain map $f\colon A\to B$ induces a chain map in cohomology
\[
H^k(f)\colon H^k(A)\to H^k(B),\quad H^k(f)(a+\mathrm{im} d_A^{k-1})=f(a)+\mathrm{im} d_B^{k-1}
\]
Any object $X$ of $\fr A$ can be regarded as a complex $\dots\to 0\to 0\to X\to 0\to 0\to\dots$ concentrated in degree $0$ (i.e., with $X$ in the $0$th position and $0$ elsewhere). This complex has obviously trivial cohomology outside degree $0$, and its $0$th cohomology is isomorphic to $X$.
\begin{defn}
Given two morphisms of complexes $f, g\colon A \to B$, a \emph{chain homotopy} from $f$ to $g$ is a collection of maps $(h^k \colon A^k \to B^{k - 1})_k$ (not necessarily a chain map) such that
$f^k - g^k = d_B^{k - 1} h^k + h^{k + 1} d_A^k$:
\[
\xymatrix{
\dots\ar[r] & A^{k-1}\ar^-{d_A^{k-1}}[rr]\ar_-{f^{k-1}}[d]<-2pt>\ar^-{g^{k-1}}[d]<2pt> && A^k\ar^-{d_A^k}[rr]\ar_{h^k}[dll]\ar_-{f^{k}}[d]<-2pt>\ar^-{g^{k}}[d]<2pt> && A^{k+1}\ar^-{d_A^{k+1}}[r]\ar_{h^{k+1}}[dll]\ar_-{f^{k+1}}[d]<-2pt>\ar^-{g^{k+1}}[d]<2pt> &\dots\\
\dots\ar[r] & B^{k-1}\ar_-{d_B^{k-1}}[rr] && B^k\ar_-{d_B^k}[rr] &&B^{k+1}\ar_-{d_B^{k+1}}[r] &\dots
}
\]
In this case $f$ and $g$ are said to be \emph{(chain) homotopic} and this is denoted $f\sim g$.
\end{defn}
The relation $\sim$ is an equivalence relation on every $\hom_{\kom(\fr A)}(A,B)$. We can thus form a new category factoring it out.
\begin{defn}
The \emph{homotopy category of complexes over} $\fr A$ is the category $\fr K(\fr A)$ defined as follows:
\begin{description}
\item[Objects] the same as the objects of $\kom(\fr A)$.
\item[Morphisms] ``morphisms of complexes modulo homotopy'', that is,
\[
\hom_{\fr K(\fr A)}(A,B)=\hom_{\kom(\fr A)}(A,B)/\sim
\]
\end{description}
\end{defn}
As a consequence of the definition, an isomorphism in $\fr K(\fr A)$ is [represented by] a homotopy equivalence, that is a map $f \in\hom_{\kom(\fr A)}(A,B)$ for which there is another map $g\in\hom_{\kom(\fr A)}(B,A)$ such that $f\circ g \sim Id_B$ and $g \circ f \sim Id_A$. The most significant example of homotopy equivalence is the augmentation map from a projective resolution to the object it resolves. Hence, in $\fr K(\fr A)$ an object of $\fr A$ is isomorphic to all its projective resolutions.

It is not difficult to see that homotopic morphisms induce the same map in cohomology. It follows that a homotopy equivalence induces an isomorphism in cohomology. In general, however, the converse is not true, so the maps satisfying the latter property deserve their own name.
\begin{defn}
A morphism in $\kom(\fr{A})$ or in $\fr K(\fr A)$ is a \emph{quasi-isomorphism} if its induced map in cohomology is an isomorphism.
\end{defn}
The derived category of $\fr A$ is obtained from $\fr K(\fr A)$ in the same way as a localization of an integral domain, that is, adding formal inverses to a suitable class of elements. In the present context, one adds formal inverses to quasi-isomorphisms of $\fr K(\fr A)$.
\begin{defn}
The \emph{derived category of} $\fr A$ is the category $\fr D(\fr A)$ defined as follows:
\begin{description}
\item[Objects] the same as the objects of $\kom(\fr A)$.
\item[Morphisms] $\hom_{\fr D(\fr A)}(A,B)$ is made of equivalence classes of \emph{roofs}
\[
\xymatrix{
&X\ar_-{\alpha}[dl]\ar^-{\phi}[dr]&\\
A&&B
}
\]
with $\phi\in\hom_{\fr D(\fr A)}(X,B)$ a morphism and $\alpha\in\hom_{\fr D(\fr A)}(A,X)$ a quasi-isomorphism, where two roofs $A\stackrel{\alpha}{\leftarrow}X\stackrel{\phi}{\to}B$ and $A\stackrel{\beta}{\leftarrow}Y\stackrel{\psi}{\to}B$ are equivalent if and only if there is a third roof $X\stackrel{\kappa}{\leftarrow}Z\stackrel{\lambda}{\to}Y$ forming a commutative diagram
\[
\xymatrix{
&&Z\ar_-{\kappa}[dl]\ar^-{\lambda}[dr]&&\\
&X\ar_-{\alpha}[dl]\ar_-{\phi}[drrr]&&Y\ar^-{\beta}[dlll]\ar^-{\psi}[dr]&\\
A&&&&B
}
\]
\end{description}
\end{defn}
The derived category enjoys a universal property: it is the most general category in which quasi-isomorphisms of complexes become isomorphisms. Namely, let $\fr Q\colon\kom(\fr A)\to\fr D(\fr A)$ be the functor that maps every object to itself and the morphism $A\stackrel{\phi}{\to}B$ to the equivalence class of the roof $A\stackrel{Id}{\leftarrow}A\stackrel{\phi}{\to}B$. Then for any functor $\fr F\colon\kom(\fr A)\to\fr C$ transforming quasi-isomorphisms into isomorphisms there is a unique functor $\fr G\colon\fr D(\fr A)\to\fr C$ such that $\fr F=\fr G\circ\fr Q$.

For certain purposes it is useful to deal only with the complexes $A$ that are \emph{bounded from above} ($A^k=0$ for $k\gg 0$), \emph{bounded from below} ($A^k=0$ for $k\ll 0$), or just \emph{bounded} ($A^k=0$ for $k\gg 0$ and $k\ll 0$). The respective full subcategories of $\kom(\fr A)$ are usually denoted $\kom^-(\fr A)$, $\kom^+(\fr A)$, $\kom^b(\fr A)$. The ``derivation'' process can be performed on these subcategories as well, giving birth to the categories
\begin{equation}\label{eq:bounded cat}
\xymatrix{
\mbox{complexes} & \mbox{homotopy} & \mbox{derived}\\
\kom^-(\fr A)\ar@{~>}[r] & \fr K^-(\fr A)\ar@{~>}[r] & \fr D^-(\fr A)\\
\kom^+(\fr A)\ar@{~>}[r] & \fr K^+(\fr A)\ar@{~>}[r] & \fr D^+(\fr A)\\
\kom^b(\fr A)\ar@{~>}[r] & \fr K^b(\fr A)\ar@{~>}[r] & \fr D^b(\fr A)
}
\end{equation}
The derived categories in the last column admit the equivalent descriptions as the full subcategories of $\fr D(\fr A)$ consisting of the complexes $A$ with $H^k(A)=0$ for $k\gg 0$, with $H^k(A)=0$ for $k\ll 0$, and with $H^k(A)=0$ for $k\gg 0$ and $k\ll 0$, respectively.

In order to simplify the writing of some statements, we will call K-categories over $\fr A$ all the categories made of chain comlexes over $\fr A$, that is, $\kom(\fr A)$, $\fr K(\fr A)$, $\fr D(\fr A)$ and their bounded subcategories introduced in Diagram \eqref{eq:bounded cat}.

\section{Triangles}\label{sec:triangles}
Derived categories are additive but not abelian (cf. \cite[Section III.3]{gelman}), hence the concept of exact sequence does not make sense in them. Nevertheless, relaxing the requirements a bit, it is possible to get a quite powerful deputy tool: the aim of this section is to introduce it.

Let $\fr C$ be a K-category over $\fr A$ and fix $n\in\Z$. For any complex $A$ in $\fr C$, one can consider the translated complex $A[n]$ defined by 
\[
A[n]^k=A^{n+k}, d_{A[n]}^k=(-1)^nd_A^{n+k}.
\]
Also, for any morphism $\phi\colon A\to B$, one can define the translated morphism $\phi[n]\colon A[n]\to B[n], \phi[n]^k=\phi^{n+k}$. This defines an autoequivalence of categories
\[
T^n\colon\fr C\to\fr C,\quad T^n(A)=A[n],\quad T^n(\phi)=\phi[n]
\]
called \emph{$n$-shift} or \emph{$n$-translation}.

To any morphism $\phi\colon A\to B$ in $\kom(\fr A)$ one can associate two complexes:
\begin{enumerate}
\item the \emph{cone} $C(\phi)$ of $\phi$, defined by
\begin{eqnarray*}
C(\phi)^k&=&A[1]^k\oplus B^k,\\
\quad d_{C(\phi)}^k(a^{(k+1)},b^{(k)})&=&\left(-d_A(a^{(k+1)}),\phi(a^{(k+1)})\!+\!d_B(b^{(k)})\right);
\end{eqnarray*}
\item the \emph{cylinder} $Cyl(\phi)$ of $\phi$, defined by
\begin{eqnarray*}
Cyl(\phi)^k&=&A^k\oplus A[1]^k\oplus B^k,\\
d_{Cyl(\phi)}^k\!(\!a^{(k)}\!,\!a^{(k\!+\!1)}\!,\!b^{(k)}\!)\!\!\!\!\!&=&\!\!\!\!\!\Big(\!d_A(\!a^{(k)}\!)\!-\!a^{(k\!+\!1)}\!,\, -d_A(\!a^{(k\!+\!1)}\!)\!,\,\phi(\!a^{(k\!+\!1)}\!)\!+\!d_B(\!b^{(k)}\!)\!\Big).
\end{eqnarray*}
\end{enumerate}
Clearly one can consider the cone and the cylinder of a morphism also as objects in the homotopy category of complexes or in the derived category as well.

\begin{defn}
A \emph{triangle} in $\fr C$ is a diagram of objects and morphisms in $\fr C$
\[
A\stackrel{\phi}{\to}B\stackrel{\chi}{\to}C\stackrel{\psi}{\to}A[1].
\]
\end{defn}
The archetypical example of triangle is the \emph{cylinder-cone} sequence of a morphism $\phi\in\hom_{\kom(\fr A)}(A,B)$. It is
\[
A\stackrel{\varphi}{\to}Cyl(\phi)\stackrel{\pi}{\to}C(\phi)\stackrel{\delta}{\to}A[1],
\]
where
\begin{eqnarray*}
 \varphi(a^{(k)})&=&(a^{(k)},0,0),\\
 \pi(a^{(k)},a^{(k+1)},b^{(k)})&=&(a^{(k+1)},b^{(k)}),\\
 \delta(a^{(k+1)},b^{(k)})&=&a^{(k+1)}.
\end{eqnarray*}
\begin{defn}
A \emph{morphism of triangles} is a commutative diagram
\[
\xymatrix{
A\ar^-{\phi}[r]\ar^-{\rho}[d] & B\ar^-{\chi}[r]\ar^-{\sigma}[d] & C\ar^-{\psi}[r]\ar^-{\tau}[d] & A[1]\ar^-{\rho[1]}[d]\\
D\ar^-{\lambda}[r] & E\ar^-{\mu}[r] & F\ar^-{\nu}[r] & A[1];
}
\]
it is an \emph{isomorphism} if $\rho$, $\sigma$, $\tau$ are isomorphisms (in their category!).

Finally, a triangle is \emph{distinguished} if it is isomorphic to the cylinder-cone sequence
\[
A\stackrel{\varphi}{\to}Cyl(\phi)\stackrel{\pi}{\to}C(\phi)\stackrel{\delta}{\to}A[1]
\]
of some $\phi\in\hom_{\kom(\fr A)}(A,B)$.
\end{defn}
The fact that distinguished triangles in derived categories are a generalization of exact sequences is expressed in the following result, that we have already used.
\begin{prop}[\cite{gelman}, Proposition III.5]
A short exact sequence in $\kom(\fr A)$ is isomorphic in $\fr D(\fr A)$ (that is, quasi-isomorphic in $\kom(\fr A)$) to a distinguished triangle.
\end{prop}
The fact that they are a good substitute for exact sequences is stated in the 
\begin{thm}[\cite{gelman}, Theorem III.6]
Let
\[
\xymatrix{
A\ar^-{\phi}[r] & B\ar^-{\chi}[r] & C\ar^-{\psi}[r] & A[1]
}
\]
be a distinguished triangle in $\fr D(\fr A)$. Then the sequence
\[
\xymatrix{
\dots\ar[r] & H^k(A)\ar^-{H^k(\phi)}[r] & H^k(B)\ar^-{H^k(\chi)}[r] & H^k(C)\ar^-{H^k(\psi)}[r] & H^{k+1}(A)\ar^-{H^{k+1}(\phi)}[r] & \dots\\
}
\]
is exact.
\end{thm}

Given a functor $\fr F\colon\fr A\to\fr B$ between abelian categories, one can extend it to chain complexes by making the extension act componentwise. Since this extension preserves the homotopy between morphisms, it induces a functor $\fr K(\fr F)\colon\fr K(\fr A)\to\fr K(\fr B)$. If $\fr F$ is exact, $\fr K(\fr F)$ maps quasi-isomorphisms to quasi-isomorphisms, so it induces a \emph{derived} functor $\fr D(\fr F)\colon\fr D(\fr A)\to\fr D(\fr B)$ and this one maps distinguished triangles to distinguished triangles. A functor between derived categories with this property is called \emph{exact}.
The same course holds in the subcategories of bounded complexes.

\section{Triangulated categories}
The machinery of distinguished triangles was developed in parallel to the concept of derived category by Jean-Louis Verdier in his PhD thesis. He realized that the properties of the  distinguished triangles in a derived category give birth to an interesting structure that is worth of independent studies. So he provided general axioms to encode the behaviour of distinguished triangles.
\begin{defn}
A \emph{triangulation} on an additive category $\fr C$ is an additive self-equivalence $T\colon\fr C\to\fr C$ together with a collection $\mathcal T$ of triangles
\[
 A\stackrel{\phi}{\to} B\stackrel{\chi}{\to} C\stackrel{\psi}{\to} T(A),
\]
called the \emph{distinguished triangles}, such that the following axioms hold.
\begin{description}
 \item[TR1] \begin{itemize}
             \item Any triangle isomorphic to a distinguished one is itself distinguished.
             \item Any morphism $A\stackrel{\phi}{\to}B$ can be completed to a distinguished triangle $A\stackrel{\phi}{\to} B\stackrel{\chi}{\to} C\stackrel{\psi}{\to} T(A).$
             \item The triangle $A\stackrel{Id}{\to}A\to 0\to T(A)$ is distinguished.
            \end{itemize}
 \item[TR2] A triangle $A\stackrel{\phi}{\to} B\stackrel{\chi}{\to} C\stackrel{\psi}{\to} T(A)$ is distinguished if and only if the triangle $B\stackrel{\chi}{\to} C\stackrel{\psi}{\to} T(A)\stackrel{-T(\phi)}{\to} T(B)$ is distinguished.
 \item[TR3] Any diagram
 \[
 \xymatrix{
A\ar^-{\phi}[r]\ar^-{\rho}[d] & B\ar^-{\chi}[r]\ar^-{\sigma}[d] & C\ar^-{\psi}[r] & A[1]\ar^-{\rho[1]}[d]\\
D\ar^-{\lambda}[r] & E\ar^-{\mu}[r] & F\ar^-{\nu}[r] & A[1]
}
 \]
in which the rows are distinguished and $\sigma\phi=\lambda\rho$ can be completed, through a morphism $\tau\colon C\to F$, to a morphism of triangles.
 \item[TR4] Given distinguished triangles
\[
\begin{array}{c}
    A \stackrel{\alpha}{\rightarrow} B \stackrel{\beta}{\rightarrow} G \stackrel{\gamma}{\rightarrow} T(A),\\ 
    B \stackrel{\delta}{\rightarrow} C \stackrel{\epsilon}{\rightarrow} Z \stackrel{\zeta}{\rightarrow} T(B),\\
    A \stackrel{\delta\alpha}{\rightarrow} C \stackrel{\theta}{\rightarrow} I \stackrel{\iota}{\rightarrow} T(A),
\end{array}
\]
there exists a distinguished triangle
\[
    G \stackrel{\phi}{\rightarrow} I \stackrel{\chi}{\rightarrow} Z \stackrel{\psi}{\rightarrow} T(G)
\]
such that
\[
 \xymatrix@C-4pt{
 \epsilon=\chi\theta, & \gamma=\iota\phi, & \psi=T(\beta)\zeta, & \zeta\chi=T(\alpha)\iota, & \phi\beta=\theta\delta.
 }
\]
\end{description}
A \emph{triangulated category} is an additive category $\fr C$ equipped with a triangulation $(T, \mathcal T)$ (we will denote it by the same symbol as the underlying additive category, if there is no danger of confusion).
\end{defn}
In Sections IV.1 and IV.2 of \cite{gelman} it is proved that for any abelian category $\fr A$ the $1$-translation functor together with the distinguished triangles we defined in Section \ref{sec:triangles} constitute actual triangulations on each of $\fr K(\fr A)$, $\fr D(\fr A)$ and (since cones and cylinders of morphisms of somehow bounded complexes are bounded of the same type) their bounded variants \eqref{eq:bounded cat}. These triangulations shall be named \emph{canonical} in order to underline their prominence and naturality.

An additive functor $\fr F\colon (\fr C,T,\mathcal T)\to (\fr C',T',\mathcal T')$ between triangulated categories is \emph{graded} if there is a natural isomorphism $\fr F\circ T\cong T'\circ\fr F$. It is \emph{$\delta$-exact}, $\delta\in\{\pm1\}$, if it is graded and for every distinguished triangle $A\stackrel{\phi}{\to} B\stackrel{\chi}{\to} C\stackrel{\psi}{\to} T(A)$ the triangle 
\[
\fr F(A)\stackrel{\fr F(\phi)}{\to} \fr F(B)\stackrel{\fr F(\chi)}{\to} \fr F(C)\stackrel{\delta\cdot\fr F(\psi)}{\to} T'(\fr F(A))
\]
is distinguished. Similar definitions hold in case $\fr F$ is a contravariant functor: $\fr F$ graded means $\fr F\circ T\cong (T')^{-1}\circ\fr F$ and $\fr F$ $\delta$-exact means for every distinguished triangle $A\stackrel{\phi}{\to} B\stackrel{\chi}{\to} C\stackrel{\psi}{\to} T(A)$ the triangle 
\[
\fr F(C)\stackrel{\fr F(\chi)}{\to} \fr F(B)\stackrel{\fr F(\phi)}{\to} \fr F(A)\stackrel{\delta\cdot T'(\fr F(\psi))}{\to} T'(\fr F(C))
\]
is distinguished.

\section{Dualities and unitary functors}
\begin{defn}\label{def:duality and self-duality}
Let $\fr C$ be a category. A couple $(\slot^\sharp,\varpi)$ consisting of a contravariant functor $\slot^\sharp\colon\fr C\to\fr C$, together with a natural isomorphism $\varpi\colon Id_{\fr C}\to\slot^{\sharp\sharp}$ satisfying
\begin{equation}
\varpi_A^\sharp\circ\varpi_{A^\sharp}= Id_{A^\sharp}
\end{equation}
for all objects $A$ of $\fr C$, is called a \emph{duality}.

Let $(\fr C,\slot^\sharp,\varpi)$ be a category with duality. Then any morphism $\phi\colon A \to B^\sharp$ in $\fr C$ has an adjoint map
\[
 \phi^\sharp_\varpi:=\phi^\sharp\circ\varpi_B\colon B\to A^\sharp.
\]
A map $\phi\colon A\to A^\sharp$ satisfying $\phi^\sharp_\varpi=\phi$ is called \emph{self-adjoint}. A \emph{self-duality} is a self-adjoint isomorphism; in this case, each of its domain and codomain is a \emph{self-dual} object.
\end{defn}

When dealing with triangulated categories, one requires dualities to respect also the extra bit of structure given by the translation functor, namely
\begin{defn}\label{def:triang duality and self-duality}
Let $\fr C$ be a triangulated category. A couple $(\slot^\sharp,\varpi)$ consisting of a $\delta$-exact contravariant functor $\slot^\sharp\colon\fr C\to\fr C$, together with a natural isomorphism $\varpi\colon Id_{\fr C}\to\slot^{\sharp\sharp}$ satisfying
\begin{eqnarray}
 \varpi_A^\sharp\circ\varpi_{A^\sharp}&=& Id_{A^\sharp},\\
 \varpi_{T(A)} &=& T_{\varpi(A)},
\end{eqnarray}
for all objects $A$ of $\fr C$, is called a \emph{$\delta$-duality}.
In this case, by a \emph{self-duality} we mean an isomorphism of distinguished triangles
\[
\xymatrix{
A\ar^-{\phi}[r]\ar^-{\rho}[d] & B\ar^-{\chi}[r]\ar^-{\sigma}[d] & C\ar^-{\psi}[r]\ar^-{\tau}[d] & A[1]\ar^-{\rho[1]}[d]\\
C^\sharp\ar^-{\chi^\sharp}[r] & B^\sharp\ar^-{\phi^\sharp}[r] & A^\sharp\ar^-{\delta\cdot\psi^\sharp[1]}[r] & C^\sharp[1]
}
\]
such that $\sigma^\sharp\circ\varpi_B=\sigma$ and $\tau=\rho^\sharp\circ\varpi_C$. Each row of the above diagram is a \emph{self-dual} distinguished triangle.
\end{defn}
See \cite{balmer} for more details.

\begin{defn}
Let $(\mathfrak{C}, ^\sharp, \varpi)$ and $(\mathfrak{D}, ^\flat, \varrho)$ be two triangulated categories with duality. A \emph{unitary functor} \[(\mathfrak G, \varsigma)\colon (\mathfrak{C}, \slot^\sharp, \varpi)\to (\mathfrak{D}, \slot^\flat, \varrho)
\]
is a covariant $(+1)$-exact functor $\mathfrak C\to\mathfrak D$, together with a natural isomorphism of contravariant functors $\varsigma\colon \fr G(\slot^\sharp)\to\fr G(\slot)^\flat$ making the diagrams
\begin{equation}\label{eq:diagram for unitary}
\xymatrix{
\fr G(A)\ar[rr]^-{\fr G(\varpi(A))}\ar[d]_{\varrho(\fr G(A))} && \fr G(A^{\sharp\sharp})\ar[d]^{\varsigma(A^\sharp)}\\
\fr G(A)^{\flat\flat}\ar[rr]_-{\varsigma(A)^\flat} && \fr G(A^\sharp)^\flat
}
\end{equation}
commute for all objects $A$ in $\fr C$.
\end{defn}
In particular, a unitary functor takes self-dual distinguished triangles to self-dual distinguished triangles.

\section{The Restriction functor}
Let $R$ be a ring and $A$ an $R$-algebra. Then $\fr P\kom(A)$ shall denote the category whose objects are bounded chain complexes of finitely generated projective left $A$-modules and whose morphisms are chain maps. The derivation process of Section \ref{sec:derived} can be performed seamlessly in order to obtain the corresponding homotopy category $\fr{PK}(A)$ and finally the derived category $\fr P(A)$. The category $\fr P\kom(A)$ contains as objects the mapping cylinder and the mapping cone of all its morphisms, hence $\fr{PK}(A)$ and $\fr P(A)$ come equipped with a preferred triangulation, that is the restriction of the canonical triangulation introduced in Section \ref{sec:triangles}.

A particular instance of the preceding constrution involves as ingredients the ring $\Z$ and the group algebra $\Z G$ of some group $G$. Group algebras have an additional piece of structure, namely the canonical antipode $S\colon\Z[G]\to \Z[G], S(g)=g^{-1}$ introduced in Remark \ref{rem:involution}. This serves to define a duality on $\fr P(\Z G)$, as follows.
If $P$ is a finitely generated projective left $\Z G$-module, $P^\ast:=\hom_{G}(P,\Z G)$ is a finitely generated projective right $\Z G$-module. Twisting the action on $P^*$ by the antipode we get a finitely generated projective left $\Z G$ module
\[
 P\ox=\hom_G^S(P,\Z G)=\{\alpha\colon P\to\Z G\mid\alpha(a\ldotp p)=\alpha(p)S(a)\}.
\]
Extending the functor $\slot\ox$ componentwise to $\fr P\kom(\Z G)$ defines a contravariant $(+1)$-exact functor, still denoted by $\slot\ox$, given by
\[
 P\ox_k=\hom_G^S(P^{-k},\Z G),\quad d_k^{P\ox}(p_{(k)}^*)(p^{(1-k)})=p_{(k)}^*(d_P^{1-k}(p^{(1-k)}))
\]
where $\Z G$ is concentrated in degree $0$.
This functor passes to the derived category $\fr P(\Z G)$ (we will still use the notation $\slot\ox$) and, together with the natural isomorphism
\[
 \varpi\colon Id_{\fr P(\Z G)}\to\slot\ox\!\, \ox,\quad \varpi_{P}(p^{(k)})(q^*_{(-k)})=S(q^*_{(-k)}(p^{(k)})),
\]
forms a $(+1)$-duality on $\fr P(\Z G)$ (a proof of this will be presented in \cite{latest}).

In the following, a category of the form $\fr P(\Z G)$ will be intended to be equipped with the triangulation and the duality just mentioned.

\begin{prop}\label{prop:res unitary}
Let $G$ be a group and $U\leq G$ a subgroup of finite index. The functor $\res^G_U\colon\fr P(\Z G)\to\fr P(\Z U)$ is unitary.
\end{prop}
\begin{proof}
Let $P$ be an object in $\fr P(\Z G)$. By a particular instance of Eckmann-Shapiro Lemma (sometimes called Nakayama relations, cf \cite[\s~2.8]{benson}), one has the natural isomorphism of abelian groups
\[
\begin{array}{rcl}
\Phi\colon\res^G_U\left(\hom^S_G(P,\hom^S_U(\Z G,\Z U)\right) & \to & \hom^S_U\left(\Z U\otimes_G P,\Z U\right)\\
\Phi(\alpha) & \colon & u\otimes x\mapsto\alpha(x)(u).
\end{array}
\]
But $\Phi$ is also compatible with the $U$-action, hence it is an isomorphism of $U$-modules. Moreover, $\Z U\otimes_GP$ may be identified with $\res^G_U(P)$. Finally, since the index $[G:U]$ is finite, one has the isomorphism
\[
\Psi\colon\Z G\to\hom_U(\Z G,\Z U),\quad \Psi(g)(h)=hg\delta_{\{hg\in U\}}
\]
where
\[
 \delta_{\{x\in U\}}=\begin{cases}
                      1 & x\in U\\
                      0 & x\notin U.
                     \end{cases}
\]
Summing up, there is a natural isomorphism of functors
\[
\begin{array}{rcl}
\varsigma\colon\res^G_U(\slot\ox) & \to & \res^G_U(\slot)\ox\\
\varsigma_P(\chi) & \colon & u\otimes_G x\mapsto\Psi\circ\chi(x)(u).
\end{array}
\]

Then Diagram \eqref{eq:diagram for unitary} becomes
\begin{equation}
 \xymatrix{
\res^G_U(P)\ar[d]_-{\varpi_{\res^G_U(P)}^U(\res^G_U)}\ar[r]^-{\res^G_U(\varpi_P^G)} & \res^G_U(\hom^S_G(\hom^S_G(P,\Z G),\Z G))\ar[d]^-{\varsigma_{\slot\ox}}\\
\hom^S_U(\hom^S_U(\res^G_U(P),\Z U),\Z U)\ar[r]_-{\varsigma_{\slot}\ox} & \hom^S_U(\res^G_U(\hom^S_G(P,\Z G)),\Z U)
}
\end{equation}
and (simplifying the notation a little and using the identification $\res^G_U(\slot)=\Z U\otimes_G\slot$) we need to verify that $\varsigma(1\otimes_G\varpi^G(x))=\varsigma\ox(\varpi^U(x))$ for all $x\in P$. On one side,
\[
\begin{array}{rcl}
\varsigma(\hskip-1pt 1\!\otimes_{_G}\!\varpi^G\!(\hskip-0.9pt x\hskip-0.9pt)\!\hskip0.1pt)\colon\!\res^G_U(P\ox) & \to &\Z U\\
%&&\\
\varsigma(1\!\otimes_{_G}\!\varpi^G(x))(u\otimes\alpha)\!&=&\!\Psi\Big(\varpi^G(x)(\alpha)\Big)(u)=\Big(\Psi(S(\alpha(x)))\Big)(u)=\\
\!&=&\!uS(\alpha(x))\delta_{\{uS(\alpha(x))\in U\}}=uS(\alpha(x))\delta_{\{\alpha(x)\in U\}}.
\end{array}
\]
On the other,
\[
\begin{array}{rcl}
\varsigma\ox(\varpi^U\!(x))\colon\res^G_U(P\ox) & \to &\Z U\\
%&&\\
\varsigma\ox(\varpi^U(x))(u\otimes\alpha)\!&=&\!\varpi^U(x) \Big(\varsigma(u\otimes_G\alpha)\Big)=u\cdot\varpi^U(x)(\varsigma(1\otimes_G\alpha))=\\
\!&=&\!u\cdot S(\varsigma(1\otimes_G\alpha)(x))=u\cdot S\Big((\Psi\circ\alpha)(x)(1)\Big)=\\
\!&=&\!u\cdot S(\alpha(x)\delta_{\{\alpha(x)\in U\}})=uS(\alpha(x))\delta_{\{\alpha(x)\in U\}}.
\end{array}
\]
\end{proof}
\titleformat
{\chapter} % command
[display] % shape
{\mdseries\huge\scshape} % format
{} % label
{0ex} % sep
{
    \centering
}
[
\rule{\textwidth}{1pt}
\vspace{1ex}
]
\bibliographystyle{amsplain}
\cleardoublepage
\phantomsection
\addcontentsline{toc}{chapter}{Bibliography} 
\bibliography{thesis}
\end{document}